\documentclass[11pt,a4paper,reqno]{amsart}

\usepackage{amsmath,amssymb,amsfonts}

\usepackage[colorlinks=true]{hyperref}

\newtheorem{theorem}{Theorem}[section]

\newtheorem{lemma}[theorem]{Lemma}

\theoremstyle{definition}

\newtheorem{remark}{Remark}

\newtheorem{example}[theorem]{Example}

\usepackage{amsthm}

\DeclareMathOperator{\dt}{dt}
\newcommand{\R}{\mathbb{R}}
\newcommand{\N}{\mathbb{N}}
\newcommand{\Z}{\mathbb{Z}}

\newcommand{\F}{\mathcal F}

\newcommand{\lw}{(\ell,\w)}

   \DeclareMathOperator{\Id}{Id}
   \DeclareMathOperator{\e}{e}

\usepackage{dsfont,mathrsfs}
      \newcommand{\T}{\mathbb{Z}}
      \newcommand{\cB}{\mathcal{B}}
    \newcommand{\cP}{\mathcal{P}}
   
\newcommand{\prts}[1]{\left(#1\right)}

\newcommand{\x}{\times}

\newcommand{\Sgm}{\Sigma}

\newcommand{\w}{\omega}
\newcommand{\W}{\Omega}
\newcommand{\TW}{\tilde{\W}}
\newcommand{\WW}{\W^+}
\newcommand{\TWW}{\tilde{\W}^+}

\usepackage{enumitem}
   \newcommand{\lb}{label}
   \newcommand{\lm}{leftmargin}

\title[Admissibility and generalized nonuniform dichotomies for NRDS]
{Admissibility and generalized nonuniform dichotomies for nonautonomous Random Dynamical Systems}

\author[Davor Dragi\v{c}evi\'{c}, C\'{e}sar M. Silva, Helder Vilarinho]{Davor Dragi\v{c}evi\'{c}$^1$, C\'{e}sar M. Silva$^2$ and Helder Vilarinho$^{2,*}$}

\address{$^1$ Faculty of Mathematics\\
University of Rijeka\\
Radmile Matej\v{c}i\'{c} 2, 51000 Rijeka\\
Croatia}
\email{ddragicevic@math.uniri.hr}

\address{$^2$ Centro de Matem\'atica e Aplica\c{c}\~oes\\
Universidade da Beira Interior\\
6201-001 Covilh\~a\\
Portugal}
\email{csilva@ubi.pt}
\email{helder@ubi.pt}
\urladdr{www.helder.ubi.pt}



\thanks{$^*$ Corresponding author.}

\begin{document}
	\maketitle

	\begin{abstract}
		
		In this paper, we
  introduce generalized dichotomies for non\-au\-tonomous random linear dynamical systems acting on arbitrary Banach spaces, and obtain their complete characterization  in terms of an appropriate admissibility property. These generalized dichotomies are associated to growth rates satisfying  mild conditions and they include the standard exponential behaviour as a very particular case. As a nontrivial application, we establish the robustness property of such dichotomies under small (linear) perturbations.
  
		
	\end{abstract}
	
 \section{Introduction}

 The study of admissibility and its relation with asymptotic properties of dynamical systems has a long history that goes back to the pioneering works of Perron~\cite{Per} and Li~\cite{Li} for continuous and discrete time, respectively. Subsequently, a systematic approach was initiated in the landmark works of Massera and Sch\"affer~\cite{MS1, MS2} in the case of continuous time, and by Coffman and Sch\"affer~\cite{CS} in the case of discrete time.

 Focusing for convenience on the case of discrete time, let us consider a linear nonautonomous difference equation 
 \begin{equation}\label{lde}
 x_{n+1}=A_n x_n \quad n\in J,
 \end{equation}
 on an arbitrary Banach space $X$, where $(A_n)_{n\in J}$ is a sequence of bounded operators on $X$ and $J\in \{\Z, \Z^+\}$. Let $\mathcal Y_1$ and $\mathcal Y_2$ be Banach spaces whose elements are  sequences $(x_n)_{n\in J}\subset X$. We say that the pair $(\mathcal Y_1, \mathcal Y_2)$ is \emph{admissible} for~\eqref{lde} if for each sequence $(y_n)_{n\in J}\in \mathcal Y_2$, there exists (a unique) $(x_n)_{n\in J}\in \mathcal Y_1$ such that 
 \[
 x_{n+1}=A_n x_n+y_n, \quad n\in J.
 \]
 In the above mentioned works, Li (resp. Coffman and Sch\"affer) obtained characterization of exponential stability (resp. exponential dichotomy) of~\eqref{lde} in terms of admissibility with respect to suitable pairs of spaces. The works of Perron and Massera and Sch\"affer considered a continuous time counterpart of~\eqref{lde}, i.e. a linear ordinary differential equation of the form $x'=A(t)x$.

 Subsequent important contributions (either for continuous or discrete time) are due to Coppel~\cite{Co},   Dalec$'$ki\u{\i} and Kre\v\i n~\cite{Dal}, as well as Henry~\cite{Hen}. For a small fraction of the more recent results we refer to~\cite{Huy,LRS, MSS2, Mi.Ra.Sc.1998,P.P.C.JFA.2010,S,SS1} and references therein. Finally, we refer to~\cite{B.D.V.AH.2018} for  a detailed survey.

 We stress that all the above mentioned works dealt with characterizations of \emph{uniform} exponential behavior in terms of admissibility. The case of nonuniform exponential dichotomies was first considered by Preda and Megan~\cite{PM}. For these dichotomies, the rates of contraction (resp. expansion) along stable (resp. unstable) directions depend on the initial time. 
 For subsequent contributions to the study of the relationship between admissibility and nonuniform exponential dichotomies we refer to~\cite{BDV0,BDV1,BDV2, Dragicevic_Zhang_Zhou_2024,MSS1,SBS,SS, WX,ZLZ,ZZ} and references therein.

 On the other hand, in the context of nonautonomous dynamics it is natural to study asymptotic behaviors which are not necessarily of exponential type. To the best of our knowledge, the works of  Muldowney~\cite{M} and Naulin and Pinto~\cite{NP} were the first to consider dichotomies in which the rates of contraction (resp. expansion) along the stable (resp. unstable) directions are not necessarily of exponential type. More recently, a systematic study of such dichotomies was initiated by Barreira and Valls~\cite{Ba.Va.2008-2} who showed that such behaviors  arise naturally  in the situations when certain generalized Lyapunov exponents are nonzero.
 
 In this direction, the first author obtained a complete characterization of nonuniform \emph{polynomial} dichotomies in terms of admissibility in both discrete and continuous time setting (see~\cite{D.MN.2019,D.2020.CPAA}). The second author~\cite{Silva_2021} has extended the results from~\cite{D.MN.2019} and obtained analogous results dealing with arbitrary growth rates satisfying a certain ``slow growth condition''. These results rely on the use of the so-called Lyapunov norms which transform nonuniform behavior into a uniform one. More recently, in~\cite{BD} an alternative approach was proposed which does not rely on the use of Lyapunov norms but in order to characterize dichotomic behavior one needs several admissibility conditions.

For linear cocycles (or linear skew-product semiflows) over topological dynamical systems (consisting of a compact metric space $\Omega$ and a homeomorphism $\sigma\colon \Omega \to \Omega$) two different approaches for characterizing uniform hyperbolicity in terms of admissibility have been proposed.  The first approach initiated by Mather~\cite{Mather} considers admissible spaces formed by maps from $\Omega$ to a Banach space on which cocycle acts. Since his original work which dealt  with derivative cocycles associated to smooth diffeomorphisms on a compact Riemannian manifold, there has been a significant progress on extending his results to the case of arbitrary (i.e. not necessarily derivative) linear cocycles acting on  arbitrary Banach spaces. We refer to~\cite{CL} for a detailed overview of this line of research.
The second  approach was proposed by Chow and Leiva in~\cite{Chow_Leiva_1995}. In their discrete-time setup admissible spaces consist of two-sided sequences in a Banach space on which cocycle acts, and the admissibility condition is posed on each trajectory of the base space $(\Omega, \sigma)$ separately. Related results dealing with characterizations of uniform hyperbolicity using the Chow-Leiva approach have been obtained in~\cite{Huy2,MSS3,SS-1,SS2}. For more recent results which combine ergodic theory tools with these techniques we refer to~\cite{DSS1,DSS2}.

We now focus on the case of linear cocycles over measure-preserving systems, which allow for the incorporation of random elements. As previously mentioned, significant progress has been made in understanding admissibility and robustness in deterministic dynamical systems, offering comprehensive insights into stability properties, whether in terms of exponential rates or other general types of dichotomies. However, the introduction of randomness into system dynamics presents new challenges. The \emph{Random Dynamical Systems (RDS)} framework extends the classical approach by incorporating stochastic elements that influence the evolution of the system, necessitating a re-evaluation of core stability concepts, such as dichotomies, admissibility, and robustness.

The foundational framework for RDS became established with Arnold’s book~\cite{Arnold_1998}. In this context, particularly in finite-dimensional phase spaces, the existence of  invariant manifolds with exponential asymptotic estimates, measured in terms of non-zero Lyapunov exponents, was discussed for RDS satisfying the integrability condition required by the Multiplicative Ergodic Theorem (MET). 
The issues of admissibility and robustness for RDS with dichotomies characterized by exponential rates, in different formulations, have been examined in several works. 

Tempered exponential dichotomies arise in the context of RDS, generalizing the concept of exponential dichotomies, where growth and decay rates in stable and unstable subspaces are affected by a (tempered) random variable. In the context when the assumptions of MET (or one of its infinite-dimensional versions) are fulfilled, the existence of a tempered exponential dichotomy for a linear RDS corresponds to nonvanishing of Lyapunov exponents~\cite{Arnold_1998,Lian_Lu_2010}.

A Perron-type result for perturbations of linear RDS exhibiting tempered exponential dichotomies was given by Barreira, Dragi\v cevi\' c, and Valls, in~\cite{Barreira_Dragicevic_Valls_2016a}, and admissibility and robustness results were subsequently provided by the same authors in~\cite{Barreira_Dragicevic_Valls_2016}. Further robustness results for tempered dichotomies are discussed in Cong's book~\cite{Cong_1997}, addressing finite-dimensional RDS, and 
in the work of Zhou, Lu, and Zhang~\cite{Zhou_Lu_Zhang_2013} for discrete-time linear RDS on Banach spaces.

We stress that in~\cite{Barreira_Dragicevic_Valls_2016a} the admissibility characterization of tempered exponential dichotomies relied on the use of Lyapunov norms. In addition, Abu Alhawala and Dragi\v cevi\' c~\cite{MD1, MD2}, as well as Dragi\v cevi\' c, Zhang and Zhou~\cite{Dragicevic_Zhang_Zhou_2023, Dragicevic_Zhang_Zhou_2024} explored the relationship between admissibility and tempered dichotomies in the case when admissible pairs are not built using Lyapunov norms. Results in the same spirit but relying on Mather's approach were obtained in~\cite{Chemnitz_Dragicevic_2024}.


With a different point of view to the exponential control along invariant subspaces, Barreira and Valls~\cite{Barreira_Valls_2014} discussed the robustness of exponential dichotomies in mean, examining how these stability structures persist under sufficiently small linear perturbations. Later, Barreira, Dragi\v cevi\' c, and Valls~\cite{Barreira_Dragicevic_Valls_2015} provided a characterization of exponential dichotomies in average in terms of admissibility in the discrete-time case, with a continuous-time counterpart discussed in~\cite{Barreira_Dragicevic_Valls_2016_flow} by same authors.

Despite the extensive study of exponential dichotomies and their variations within the RDS context, many open questions remain regarding more general forms of dichotomies. Bento and Vilarinho~\cite{Bento_Vilarinho_2021} explored the existence of invariant manifolds for RDS on Banach spaces that exhibit generalized dichotomies, both for continuous and discrete-time systems, in a formulation that includes all forms of exponential dichotomies for RDS while also addressing weaker growth rates strictly beyond the exponential case. 

The extension of questions related to admissibility, robustness, stability, and invariant subspaces, among others, to systems exhibiting general dichotomies such as $\mu$-dichotomies or $(\mu, \nu)$-dichotomies, while incorporating randomness, adds additional layers of complexity. Specifically, the system's evolution law integrates both randomness and nonautonomy, where the time-dependence of the growth rates introduces a form of nonautonomy that is not fully captured by the classical RDS formalism. To address these complexities, we consider Nonautonomous Random Dynamical Systems (NRDS) to investigate admissibility and robustness in systems evolving in Banach spaces, incorporating both randomness and time-varying nonautonomy.

In this direction, Oliveira-Sousa in~\cite{Sousa_Tese_2023}, and together with Caraballo, Carvalho and Langa in~\cite{Caraballo_Carvalho_Langa_Oliveira-Sousa_2021b},
 examined the robustness of hyperbolicity in autonomous differential equations (deterministic) under nonautonomous random/stochastic perturbations. As an application, they established the existence and continuity of random hyperbolic solutions, which laid the groundwork for their subsequent work on the existence and continuity of unstable and stable sets, leading to lower semicontinuity and topological structural stability of attractors~\cite{Caraballo_Carvalho_Langa_Oliveira-Sousa_2022b,Sousa_Tese_2023}.
As in the autonomous and nonautonomous deterministic situations, admissibility and robustness remain key for understanding the existence and stability of attractors in NRDS. For this later issue see \cite{Bates_Lu_Wang_2014, Caraballo_Langa_2003, Cherubini_Lamb_Rasmussen_Sato_2017,
Crauel_Debussche_Flandoli_1997, Crauel_Kloeden_Yang_2011, Cui_Kloeden_2018, 
Li_Li_Zuo_2023, 
Wang_2012_JDE,  
Wang_2014_SD, 
Wang_2014_NA,  
Yao_Zhang_2020, Zhang_Caraballo_Yang_2024}.

In this paper, we develop a framework for the admissibility  of NRDS that exhibit generalized nonuniform dichotomies with respect to random norms. Building on the  work of Silva \cite{Silva_2021}, we extend the concept of admissibility to this setting of discrete-time NRDS evolving in a Banach space. Our setting allows for the study of systems with nonuniform growth rates, offering a more flexible approach in a setting with random perturbations than traditional dichotomies. In particular, it includes as particular cases the deterministic exponential and $\mu$-dichotomies as well exponential dichotomies for RDS.
We address the admissibility problem for random nonuniform $\mu$-dichotomies, the robustness of these dichotomies, and their relationship with strong random nonuniform $(\mu,\nu)$-dichotomies.

The paper is structured as follows: In Section~\ref{se:NRDS and mu-dichotomies}, we introduce the necessary background on NRDS and \(\mu\)-dichotomies. Section~\ref{se:admissibility wrt random norm} is dedicated to developing the theory of admissibility with respect to random norms, presenting the main theorems and proofs. Specifically, in Theorem~\ref{T51}, we prove that an NRDS admitting a $\mu$-dichotomy with respect to a random norm exhibits admissibility. Conversely, in Theorem~\ref{T52}, we show that an NRDS with admissibility and appropriate control on the growth of orbits with respect to a random norm possesses a $\mu$-dichotomy. Section~\ref{se:robustness} addresses the robustness of $\mu$-dichotomies for NRDS. In Section~\ref{se:mu_vs_munu}, we discuss the relationship between $\mu$-dichotomies with respect to a random norm and random nonuniform $(\mu,\nu)$-dichotomies. These dichotomies incorporate an additional nonautonomous factor that influences the system's behavior over time, and they are defined with respect to the space norm.

        \section{NRDS and $\mu$-dichotomies}\label{se:NRDS and mu-dichotomies}
Consider a measure space $(\W, \F, \mu)$ and a \emph{measure-preserving dynamical system} $(\W, \Sgm, \mu, \theta)$,
	in the sense that:
	\begin{enumerate}[\lb=$\roman{*})$, \lm=10mm]
		\item $\theta \colon \Z \times \W \to \W$ is $\prts{\cP(\T) \otimes \F ,\F}$-measurable, where $\cP(\T)$ denotes the power set of $\T$; 
		\item $\theta^{n} \colon \W \to \W$ given by $\theta^{n}\w = \theta(n,\w)$ preserves
			the measure $\mu$ for all $n \in \T$;

		\item $\theta^{0} = \Id_{\W}$ and $\theta^{n+m}=\theta^{n}\circ\theta^{m}$ for
			all $n,m \in\T$.
	\end{enumerate}
	If $\mu$ is a probability measure, then $(\W, \F, \mu, \theta)$ is called a
	\emph{metric dynamical system}.
        We set $\WW=\mathbb{Z}^+ \times \W$. Let $\Theta\colon\WW\to\WW$ be given by  \[\Theta(\ell,\w)=(\ell+1,\theta\omega) \quad  (\ell,\w)\in\WW,\] 
        where for the sake of simplicity throughout the paper  we denote $\theta^1$ by $\theta$.
        A measurable \emph{nonautonomous random dynamical
	system} (NRDS) on a Banach space $X=(X, \|\cdot \|)$ over a metric dynamical system $(\W, \F, \mu,     \theta)$ with time $\Z^{+}=\Z\cap[0,+\infty)$, is a map
	\begin{equation*}
		\Phi:\T^{+}\times\T^+\times\W\times X \to X
	\end{equation*}
	such that
	\begin{enumerate}[\lb=$\roman{*})$, \lm=10mm]
        \item $(n,\w) \mapsto \Phi(n,\ell,\w,v)$ is $\prts{\cP(\T^+) \otimes \W,\cB(X)}$-measurable for every $\ell\in\T^+$ and $v \in X$, where $\cB(X)$ denotes the Borel $\sigma$-algebra on $X$;
        	
		\item $\Phi(n,\ell,\w) \colon X \to X$ given by $\Phi(n,\ell,\w) v = \Phi(n,\ell,\w,v)$
			forms a cocycle over~$\Theta$, i.e.,
			\begin{enumerate}[\lb=$\alph{*})$, \lm=6mm]
				\item $\Phi(0,\ell,\w)=\Id_{X}$ for all $\lw\in\WW$;
				\item for all
					$m,n \in\T^{+}$ and $\lw\in\WW$, $$\Phi(m+n,\ell,\w)=\Phi(m,\ell+n,\theta^n\w)\circ\Phi(n,\ell,\w).$$
			\end{enumerate}
	\end{enumerate}
 In the case when it is clear which driving system $(\W, \F, \mu, \theta)$ we consider, we will identify the pair $(\Phi,\theta)$ with $\Phi$.
Let $\mathcal L(X)$
be the Banach algebra of all bounded linear operators acting on $X$ equipped with the operator norm (which we will also denote by $\| \cdot \|$). A measurable NRDS $\Phi$ is called \emph{linear} if $\Phi(n,\ell,\w)\in \mathcal L(X)$ for all $n,\ell\in\Z^+$ and $\w\in\W$.

\begin{remark}
    In the current work, we focus on the future orbits for the discrete-time case, but the definition can be easily extended to the invertible and/or continuous-time case. Furthermore, we note that, although we assume the driving system $\theta$ to be invertible, we do not require that $\Phi(n,\ell,\w)$ is invertible.
    \end{remark}

\medskip

	Set $\R^+=[0,+\infty)$. A map $\mu\colon\Z^+\to\R^{+}$
	is said to be  a \emph{growth rate}  if
	\begin{enumerate}[\lb=$\roman{*})$, \lm=10mm]
	\item  $n\mapsto \mu_n:=\mu(n)$ is strictly increasing;
	\item	$\lim \mu_n=+\infty$;
	\item exists $\eta \ge 1$ such that $\mu_{n+1}\leq \eta \mu_n$, for $n\in \Z^+$.
	\end{enumerate}	
 Given a growth rate $\mu$, for each  $n\in\Z^+$ we set $\mu_n':
	=\mu_{n+1}-\mu_n$ and
	$\varphi_n :=\frac{\mu_n}{\mu_n'}.$ 

\medskip 
 We say that a measurable function $\mathcal{N}\colon \WW\x X\to\R^{+}$ is a \emph{random norm} if, for all $\lw\in\WW$, the function $\|\cdot\|_{\lw}:=\mathcal{N} (\lw,\cdot)$ is a norm on $X$ that is equivalent to the norm $\|\cdot\|$ of the space $X$ and, additionally, for each $v\in X$, the map $\lw \mapsto \|v\|_{\lw}$ is measurable.

 \medskip
For a subset $\TW \subset \W$, denote $\TWW = \Z^+ \times \TW$. Given a growth rate $\mu$ and a random norm $\mathcal{N}$, we say that a linear NRDS $\Phi$ admits a \textit{$\mu$-dichotomy with respect to $\mathcal{N}$} if there exists a $\theta$-invariant subset $\tilde{\Omega}$ of $\Omega$ with full $\mu$-measure, and a family of projections 
$$\textstyle P:=\left\{P_{(\ell,\omega)}: \lw \in \TWW\right\},$$
such that:

\begin{enumerate}[\lb=$P_{\arabic{*}})$, \lm=10mm]
\item\label{def:mu:dich:measu:proj} for each $\ell \in \T^+$ the map $P_{\ell}: \tilde\Omega \rightarrow \mathcal L(X)$, given by $P_\ell(\w) :=P_{\lw}$, is strongly measurable, that is, for every $v\in X$,
\[
\tilde\W\ni\w\mapsto P_{\lw} v
\]
is measurable;
\item for every $n \in \T^{+}$ and $\lw \in \TWW$ we have 
$$P_{\Theta^n \lw}^{}\Phi(n,\ell,\w)=\Phi(n,\ell,\w) P_{\lw}^{};$$
 \item  the map $\Phi(n,\ell,\w)\vert_{\ker P_{\lw}}\colon \ker P_{\lw} \to \ker P_{\Theta^n\lw}$
			is  invertible for all $n \in \T^{+}$ and all $(\ell,\w)\in\TWW$;
 \item  for all $n\in\T^+$ and $v\in X$  the map
			\begin{equation*}
				(\ell,\w) \mapsto \Phi(-n,\ell+n,\theta^n\w)Q_{\Theta^n(\ell,\w)}^{}v
			\end{equation*}
			is $\prts{\cB(\T^+) \otimes \F,\cB(X)}$-measurable, where  $Q_{\lw}= \Id_X-P_{\lw}$ and
   $\Phi(-n,\ell+n,\theta^n\w)$ stands for the inverse of
	\[
		\Phi(n,\ell,\w)\vert_{\ker P_{\lw}^{}}\colon \ker P_{\lw} \to \ker P_{\Theta^n(\ell,\w)};
	\]

\item \label{def:mu:dich}
 there exist a forward $\Theta$-invariant random variable $\lambda\colon \TWW\rightarrow (0,+\infty)$ (i.e., $\lambda(0,\w)=\lambda(\ell,\theta^\ell\w)$ for all $\lw \in \TWW$), and a $\Theta$-forward invariant
  random variable $D\colon\TWW\to[1,+\infty)$ such that, for all $\ell, n\in \Z^+$, $\w\in\TW$, 
 and $v\in X$, we have
	\begin{equation}
		\label{eq:dich-1}\|\Phi(n,\ell,\w) P_{\lw} v\|_{\Theta^n(\ell,\w)}\le 	
        D\lw\left(\frac{\mu_{\ell+n}}{\mu_\ell}\right)^{-\lambda\lw}\|v\|_{\lw}
	\end{equation}
	and
	\begin{equation}
		\label{eq:dich-2}\|\Phi(-n,\ell+n,\theta^n\w)Q_{\Theta^n(\ell,\w)}v\|_{\lw}\le 
		D\lw\left(\frac{\mu_{\ell+n}}{\mu_\ell}\right)^{-\lambda\lw}\|v\|_{\Theta^n(\ell,\w)}. 
  	\end{equation}
\end{enumerate}

\begin{remark}	Notice that~\eqref{eq:dich-2} is equivalent to say that for all
$\lw\in\TWW$, $0\leq n\leq \ell$, and $v\in X$, we have
\begin{equation}
		\label{eq:dich-2-b}
		\|\Phi(-n,\ell,\w)Q_{\lw}v\|_{(\ell-n,\theta^{-n}{\w})}\le  D\lw\left(\frac{\mu_\ell}{\mu_{\ell-n}}\right)^{-\lambda \lw}\|v\|_{\lw}. 
	\end{equation}
\end{remark}

\begin{example}\label{ex:exp}
  We say that a linear NRDS $\Phi$ admits an \textit{exponential dichotomy} with respect to a random $\mathcal N$ if it admits a $\mu$-dichotomy with $\mu_n=\e^ {n}$, converting item~\ref{def:mu:dich} above to:  
$$
 \left\|\Phi(n,\ell,\w) P_{\lw}^{}v\right\|_{\Theta^n\lw} \leq D\lw \e^ {-\lambda\left(\ell,\w\right)n}\|v\|_{\lw} 
 $$
 and
\begin{equation*}
\left\|\Phi(-n,\ell+n,\theta^n\w)Q_{\Theta^n \lw}^{}v\right\|_{\lw}\leq D\lw \e^ {-\lambda\lw n}\|v\|_{\Theta^n\lw}.
\end{equation*}
\end{example}

\begin{example}
If we consider a $\mu$-dichotomy with $\mu_n=n+1$ we are lead to a polynomial type dichotomy. Conditions in~\ref{def:mu:dich} become now:
\begin{equation*}
\|\Phi(n,\ell,\w) P_{\lw} v\|_{\Theta^n{\lw}}\le  D\lw\left(\frac{\ell+n+1}{\ell+1}\right)^{-\lambda\lw}\|v\|_{{\lw}}
	\end{equation*}
	and
	\begin{equation*}
  \|\Phi(-n, \ell+n, \theta^n \omega)
Q_{\Theta^n{\lw}}v\|_{{\lw}}\le D\lw
		 \left(\frac{\ell+n+1}{\ell+1}\right)^{-\lambda\lw}\|v\|_{\Theta^n{\lw}}.
	\end{equation*}
\end{example}

\begin{example}
By considering a $\mu$-dichotomy with $\mu_n = \log(n+1)$, we arrive at a logarithmic type dichotomy. Conditions~\eqref{eq:dich-1} and~\eqref{eq:dich-2} read now as
\begin{equation*}
	\|\Phi(n,\ell,\w) P_{\lw} v\|_{\Theta^n{\lw}} \le D\lw\left(\frac{\log(\ell+n+1)}{\log (\ell+1)}\right)^{-\lambda\lw}\|v\|_{{\lw}}
	\end{equation*}
	and
	\begin{equation*}
		\|\Phi(-n, \ell+n, \theta^n \omega)Q_{\Theta^n{\lw}}v\|_{{\lw}}\le D\lw
		 \left(\frac{\log(\ell+n+1)}{\log (\ell+1)}\right)^{-\lambda\lw}\|v\|_{\Theta^n{\lw}}.
	\end{equation*}
\end{example}

\begin{example}\label{ex:exist:dich}
Let us now provide a simple NRDS that exhibits a $\mu$-dichotomy with a prescribed growth rate $\mu$. Let $X$ be a Banach space with norm $\|\cdot\|$, $(\W,\mathcal F,\mu,\theta)$ a metric dynamical system, $\WW=\Z^+\times\W$ and $\Theta\colon\WW\to\WW$ to be given by $\Theta(\ell,\w)=(\ell+1,\theta\w)$. Consider a random variable $K\colon\WW\to [1,+\infty)$ and define a random norm $\mathcal N$ by setting $$\|\cdot\|_{\lw}=K\lw\|\cdot\|.$$ Now, consider a 
an arbitrary projection $P$ on $X$.
Additionally, let $\lambda\colon \T^+\x\Omega\rightarrow\R$ and $D\colon\Z^+\times\W\to (0,\infty)$ be $\Theta$-forward random variables.  
We set
\[
\begin{split}
\Phi^s(n,\ell,\w)=
D\lw \frac{K\lw}{K(\ell+n,\theta^n\w)}\left(\frac{\mu_{\ell+n}}{\mu_\ell}\right)^{-\lambda\lw}P
\end{split}
\]
and
\[
\begin{split}
\Phi^u(n,\ell,\w)=
    \frac{1}{D\lw}\frac{ K\lw}{K(\ell+n,\theta^n\w)}\left(\frac{\mu_{\ell+n}}{\mu_\ell}\right)^{\lambda\lw}Q,
\end{split}
\]
where $Q:=\Id_X-P$.
For all $(n,\ell,\w)\in\Z^+\times\Z^+\times\W$, we define
\[
\Phi(n,\ell,\w)=\Phi^s(n,\ell,\w)+\Phi^u(n,\ell,\w).
\]
For all
$\lw\in\WW$, $n\in \Z^+$ and $v\in X$, we have that 
\[
\begin{split}
\|\Phi(n,\ell,\w)Pv\|_{\Theta^n(\ell,\w)}&= D\lw K\lw\left(\frac{\mu_{\ell+n}}{\mu_\ell}\right)^{-\lambda\lw}\left\|Pv\right\|\\
&
\leq D\lw \|P\| \left(\frac{\mu_{\ell+n}}{\mu_\ell}\right)^{-\lambda\lw}\left\|v\right\|_{\lw}.
\end{split}
\]
Similarly, since
\[
\Phi(-n,\ell+n,\theta^n\w)Q =  D\lw\frac{K(\ell+n,\theta^n\w)}{ K\lw}\left(\frac{\mu_{\ell+n}}{\mu_\ell}\right)^{-\lambda\lw}Q,
\]
we  have
\[
\begin{split}
&\|\Phi(-n,\ell+n,\theta^n\w)Qv\|_{\lw}\\
&= D\lw K(\ell+n,\theta^n\w)\left(\frac{\mu_{\ell+n}}{\mu_\ell}\right)^{-\lambda\lw}\left\|Qv\right\|\\
&\leq D\lw \|Q\| \left(\frac{\mu_{\ell+n}}{\mu_\ell}\right)^{-\lambda\lw}\left\|v\right\|_{\Theta^n(\ell,\w)}.
\end{split}
\]
\end{example}

\begin{remark} 
Note that if we disregard the dependence on $\w$ in the definition of a $\mu$-dichotomy, we arrive at the deterministic scenario of $\mu$-dichotomies, such as those in~\cite{Silva_2021}. On the other hand, in Example~\ref{ex:exp}, if we retain the dependence on $\w$ but remove the dependence on $\ell$, we are led to nonuniform exponential dichotomies in the framework of RDS, or tempered nonuniform exponential dichotomies if we assume $D(\cdot)=D(\ell,\cdot)$ to be a tempered random variable. However, the $\mu$-dichotomies that we are considering for general NRDS require simultaneously both randomness and nonautonomy in a manner that is not achievable in the RDS setting.
\end{remark}

	\section{Admissibility with respect to a random norm}\label{se:admissibility wrt random norm}

A measurable NRDS $\Phi$ can be associated with an nonautonomous random difference equations in the following way. To start, for each $\w \in \W$ consider the random difference equation
\begin{equation}\label{eq:diff:eq}
x_{n+1}=\Phi(1,n,\theta^n\w)x_n,
\end{equation}
with $x_n\in X$, $n\in\Z^+$. The solutions $x=(x_n)$ of ~\eqref{eq:diff:eq} are the orbits $(\Phi(n,0,\w)x_0)_{n\in\Z^+}$. Now, considering some growth rate $\mu$, given $\w\in\W$ and $y=(y_n)$, with $y_0=0$, we are interested in finding  $x=(x_n)_{n\in \Z^+}\subset X$  such that  
\begin{equation*}
x_{n+1}=\Phi(1,n,\theta^{n}\w) x_n+\varphi_{n+1}^{-1}y_{n+1}, \quad n\in \Z^+.
\end{equation*}
Similarly, given a growth rate $\mu$, and a measurable map $y\colon\Z^+\x\tilde\W\x\Z^+\to X$ we look for measurable $x\colon\Z^+\x\tilde\W\x\Z^+\to X$ such that 
\begin{equation}\label{eq:ADM_0}
  x(\ell, \w, n+1)=\Phi(1, n+\ell, \theta^n \w)x(\ell, \w, n)+\varphi_{\ell+n+1}^{-1}y(\ell, \w, n+1),
  \end{equation} 
  for $(\ell, \w, n)\in \Z^+\times \tilde \Omega \times \Z^+$ where $\tilde \Omega \subset \Omega$ is a $\theta$-invariant full-measure set.
  

The concept of admissibility requires that for input $y$ there exists a (unique) output $x$ satisfying~\eqref{eq:ADM_0}, where $x$ and $y$ belong to suitable Banach spaces. 

\subsection{Technical lemma}
	Before stating our main results, we will begin with a technical lemma that
	will be very useful in some estimates. This result was stated in~\cite{Silva_2021} and the proof presented there was based on a result
	concerning the smooth global approximation of a Lipschitz function by differentiable
	functions, obtained in~\cite{Czarnecki.Rifford.TAMS.2006}.
	It turns out that this proof has a issue. Fortunately, this issue can be overcomed with a simpler proof. For the sake of completeness, we provide this proof here.

\begin{lemma}\label{prop:ESTIMATES}
Let $\mu$ be a growth rate. If $\alpha \in \ (0,1) \ \cup \ (1,+\infty)$, for each $r,s \in \Z^+$ with $r \ge s >1$, we have
\begin{equation}\label{prop:ESTIMATES-est-le-<0ne1}
 \frac{\mu_{r+1}^{1-\alpha}-\mu_s^{1-\alpha}}{1-\alpha} \le \sum_{k=s}^r \mu_k^{-\alpha}\mu_k' \le  \eta^\alpha\,\frac{\mu_{r+1}^{1-\alpha}-\mu_s^{1-\alpha}}{1-\alpha}
\end{equation}
and, for $\alpha=1$, we have
\begin{equation}\label{prop:ESTIMATES-est-=1}
 \log \frac{\mu_{r+1}}{\mu_{s}} \le \sum_{k=s}^{r} \, \frac{\mu_k'}{\mu_k} \le \eta\log \frac{\mu_{r+1}}{\mu_s}.
\end{equation}
\end{lemma}

\begin{proof}
	Let $r,s \in \N$ with $r \ge s >1$ and define
		\[
		f(t)=
		\begin{cases}
			\mu_s+(\mu_{s+1}-\mu_s)(t-s), & s\le t<s+1\\
			\mu_{s+1}+(\mu_{s+2}-\mu_{s-1})(t-s-1), & s\le t<s+1\\
			\cdots \ & \ \cdots\\
			\mu_s+(\mu_{s+1}-\mu_s)(t-r-1), & r-1\le t \le r.
		\end{cases}
		\]
		Since $f$ is increasing, $f^{-\alpha}$ is decreasing. Furthermore, $f$ is differentiable in each interval $]k,k+1[$, $k=s,\ldots,r-1$, and $f'(t)=\mu_{k+1}-\mu_k=\mu_k'$ for $t \in ]k,k+1[$, $k=s,\ldots,r-1$. Thus, since $\mu_{k+1}\le \eta\mu_k$, for all $k \in \N$, we get
		\begin{equation}\label{eq:lemma-tecnico}
			\begin{split}
				\eta^{-\alpha} \mu_k^{-\alpha}\mu_k'\le
				\mu_{k+1}^{-\alpha}\mu_k'\le 
				\int_k^{k+1} [f(t)]^{-\alpha} f'(t)\dt 
				\le \mu_k^{-\alpha}\mu_k'.
			\end{split}
		\end{equation}
		By~\eqref{eq:lemma-tecnico}, 
		\[
		\begin{split}
			\eta^{-\alpha} \sum_{k=s}^r \mu_k^{-\alpha}\mu_k'
			& \le \sum_{k=s}^r \int_k^{k+1} [f(t)]^{-\alpha} f'(t)\dt 
			\le \sum_{k=s}^r \left[\frac{[f(t)^{-\alpha+1}]}{-\alpha+1}\right]_k^{k+1} f'(t)\dt \\
			& =\sum_{k=s}^r \frac{\mu_{k+1}^{-\alpha+1}- \mu_k^{-\alpha+1}}{-\alpha+1}=\frac{\mu_{r+1}^{-\alpha+1}- \mu_s^{-\alpha+1}}{-\alpha+1}.
		\end{split}
		\]
		Thus
		\[
		\sum_{k=s}^r \mu_k^{-\alpha}\mu_k' \le \eta^\alpha \frac{\mu_{r+1}^{1-\alpha}- \mu_s^{1-\alpha}}{1-\alpha}.
		\]
		Again by~\eqref{eq:lemma-tecnico}, 
		\[
		\sum_{k=s}^r \mu_k^{-\alpha}\mu_k' \ge \sum_{k=s}^r \int_k^{k+1} [f(t)]^{-\alpha} f'(t)\dt=\frac{\mu_{r+1}^{-\alpha+1}- \mu_s^{-\alpha+1}}{1-\alpha}.
		\]
		and we obtain~\eqref{prop:ESTIMATES-est-le-<0ne1}.
	
	Letting $\alpha=1$ and taking into account $\mu_{k+1}\le \eta\mu_k$, for all $k \in \N$, we obtain
		\begin{equation}\label{eq:lemma-tecnico-2}
			\frac{\mu_k'}{\eta\mu_k} \le \frac{\mu_k'}{\mu_{k+1}}
			\le \int_k^{k+1} \frac{\mu_k'}{\mu_{k+1}}\dt
			\le \int_k^{k+1} \frac{f'(t)}{f(t)}\dt\le \int_k^{k+1} \frac{\mu_k'}{\mu_k}\dt
			=\frac{\mu_k'}{\mu_k}.
		\end{equation}
		By~\eqref{eq:lemma-tecnico-2}, we obtain
		\[
			\begin{split}
				\frac{1}{\eta} \sum_{k=s}^r \frac{\mu_k'}{\mu_k} 
				& \le \sum_{k=s}^r \int_k^{k+1} \frac{f'(t)}{f(t)}\dt = \sum_{k=s}^r 
				\left[ \ln f(t)  \right]_k^{k+1}\\ 
				& = \sum_{k=s}^r \left(\log\mu_{k+1}-\log\mu_k\right)=\log\frac{\mu_{r+1}}{\mu_s}.	
			\end{split}
		\]
		and thus
		\[
		\sum_{k=s}^r \frac{\mu_k'}{\mu_k}\le \eta\log\frac{\mu_{r+1}}{\mu_s}.	
		\]
		Again by~\eqref{eq:lemma-tecnico-2}, we obtain
		\[
			\sum_{k=s}^r \frac{\mu_k'}{\mu_k} \ge \sum_{k=s}^r \int_k^{k+1} \frac{f'(t)}{f(t)}\dt =\log\frac{\mu_{r+1}}{\mu_s},	
		\]
		and we get~\eqref{prop:ESTIMATES-est-=1}. The result follows.
		\end{proof}

\bigskip



\subsection{From dichotomies to admissibility and vice-versa}
For a full measure set $\tilde \Omega \subset \Omega$, let $\mathcal Y(\tilde \Omega)$ denote the space of all measurable maps $y\colon \Z^+ \times \tilde \Omega \times \Z^+ \to X$ such that 
\[
\|y\|_{\mathcal Y(\tilde \Omega)}:=\sup_{(\ell, \w, n)}\|y(\ell, \w, n)\|_{\Theta^n \lw}<+\infty.
\]
It is straightforward to verify that $(\mathcal Y(\tilde \Omega), \| \cdot \|_{\mathcal Y(\tilde \Omega)})$ is a Banach space. By $\mathcal Y^0(\tilde \Omega)$ we will denote the set of all $y\in \mathcal Y(\tilde \Omega)$ such that $y(\ell, \w, 0)=0$. Then, $\mathcal Y^0(\tilde \Omega)$ is a closed subspace of $\mathcal Y(\tilde \Omega)$.

Let $\Pi=(\Pi \lw)_{\lw \in \TWW}$ be a family of projections on $X$. Furthermore, let $C\colon \TWW\to (0, \infty)$ be a $\Theta$-forward invariant random variable. By $\mathcal Y_C^\Pi(\tilde \Omega)$ we will denote the space of all measurable $y\colon \Z^+\times\TW\times \Z^+\to X$ such that $y(\ell, \w, 0)\in \ker \Pi \lw$
and 
\[
\|y\|_{\mathcal Y_C^\Pi(\tilde \Omega)}:=\sup_{(\ell, \w, n)}\left (C(\ell, \w)\|y(\ell, \w, n)\|_{\Theta^n \lw}\right )<+\infty.
\]
Then, $(\mathcal Y_D^\Pi(\tilde \Omega), \| \cdot \|_{\mathcal Y_D^\Pi(\tilde \Omega)})$ is a Banach space. We stress that spaces $\mathcal Y^0(\tilde \Omega)$ and $\mathcal Y_C^\Pi(\tilde \Omega)$ depend on a random norm $\mathcal N$. However, in order to ease the notation we do not make this dependence explicit.

\begin{theorem}\label{T51}
Let $\Phi$ be a measurable linear NRDS admitting a $\mu$-dichotomy
		with respect to a random norm $\mathcal{N}$. 
  Then, there exist a $\theta$-invariant set $\tilde \Omega \subset \Omega$ of full measure  and a $\Theta$-invariant random variable $C\colon \TWW\to (0, \infty) $ such that, for each $y\in \mathcal Y^0(\tilde \Omega)$, there exists a unique $x\in \mathcal Y_C^\Pi(\tilde \Omega)$ satisfying 
  \begin{equation}\label{ADM}
  x(\ell, \w, n+1)=\Phi(1, n+\ell, \theta^n \w)x(\ell, \w, n)+\varphi_{\ell+n+1}^{-1}y(\ell, \w, n+1),
  \end{equation}
  for $(\ell, \w, n)\in \Z^+ \times \tilde \Omega \times \Z^+.$
  Here, $\Pi=(\Pi \lw)_{\lw \in \TWW}$, $\Pi \lw=P_{\lw}$, is a family of projections corresponding to the $\mu$-dichotomy of $\Phi$. 
  
  \end{theorem}

  \begin{proof}
 Take $y\in \mathcal Y^0(\tilde \Omega)$ and set
\[
x^{(s)}(\ell, \w, n)=
							\displaystyle\sum_{k=0}^{n}\varphi_{\ell+k}^{-1}\Phi(n-k,\ell+k,\theta^k\w)P_{\Theta^k{\lw}}
       y(\ell, \w, k)
\]
and 
\[
x^{(u)}(\ell, \w, n)=-\sum_{j=1}^{\infty} \varphi_{\ell+n+j}^{-1}\Phi(-j,\ell+n+j,\theta^{n+j}\w)Q_{\Theta^{n+j}{\lw}}y(\ell, \w, n+j),
\]
for $(\ell, \w, n)\in \Z^+ \times \tilde \Omega \times \Z^+$.  Using~\eqref{eq:dich-1} and Lemma~\ref{prop:ESTIMATES} we have, for each $(\ell, \w, n) \in \Z^+ \times \tilde \Omega \times \Z^+$, 
\begin{equation}\label{7-13}
			\begin{split}
				&\left\|x^{(s)}(\ell, \w, n)\right\|_{\Theta^n{\lw}}\\
				 &\le \displaystyle\sum_{k=1}^{n}\varphi_{\ell+k}^{-1}D(\ell+k,\theta^k\w)\left(\frac{\mu_{\ell+n}}{\mu_{\ell+k}}\right)^{-\lambda(\ell+k,\theta^k\w)}\|y(\ell, \w, k)\|_{\Theta^k{\lw}}
                \\
				&\le D\lw \left\|y \right\|_{\mathcal Y(\tilde \Omega)} \mu_{\ell+n}^{-\lambda\lw}\sum_{k=1}^{n}\mu_{\ell+k}'\mu_{\ell+k}^{\lambda\lw-1}
                \\
                &\le D\lw \left\|y \right\|_{\mathcal Y(\tilde \Omega)} \mu_{\ell+n}^{-\lambda\lw}\frac{\eta^{1-\lambda\lw}}{\lambda\lw}(\mu_{\ell+n+1}^{\lambda\lw}-\mu_{\ell+1}^{\lambda\lw})
                \\
                &\le {D\lw}\left\|y \right\|_{\mathcal Y(\tilde \Omega)}\frac\eta{\lambda\lw}.
                \end{split}
                \end{equation}
                Similarly, 
        \begin{equation}\label{7-14}
		\begin{split}
				&\left\| x^{(u)}(\ell, \w, n)\right\|\\
 &\le D(\ell+n,\theta^n\w)\sum_{j=1}^{\infty}
				\varphi_{\ell+n+j}^{-1}\left(\frac{\mu_{\ell+n+j}}{\mu_{\ell+n}}\right)^{-\lambda(\ell+n,\theta^n\w)}\left\|y(\ell, \w, n+j)
	 \right\|_{\Theta^{n+j}{\lw}}
                \\&\le D(\ell,\w) \left\|y \right \|_{\mathcal Y(\tilde \Omega)}\mu_{\ell+n}^{\lambda\lw} \sum_{j=1}^{\infty} \mu_{\ell+n+j}'\mu_{\ell+n+j}^{-\lambda(\ell,\w)-1}
                \\&\le D(\ell,\w) \left\|y \right \|_{\mathcal Y(\tilde \Omega)}\mu_{\ell+n}^{\lambda\lw} \lim_{r\to+\infty}
                    \frac{\eta^{\lambda\lw+1}}{-\lambda\lw}\left(\mu_{\ell+n+r+1}^{-\lambda\lw}-\mu_{\ell+n+1}^{-\lambda(\ell,\w)}\right)                
                \\&\le D(\ell,\w) \left\|y \right \|_{\mathcal Y(\tilde \Omega)}\frac{\eta^{\lambda(\ell,\w)+1}}{\lambda\lw},
                \end{split}
		\end{equation}
  for $(\ell, \w, n)\in \Z^+ \times \tilde \Omega \times \Z^+$. Let
  \[
  x(\ell, \w, n):=x^{(s)}(\ell, \w, n)+x^{(u)}(\ell, \w, n), \quad (\ell, \w, n)\in \Z^+ \times \tilde \Omega \times \Z^+.
  \]
  Observe that $x$ is measurable, $x(\ell, \w, 0)\in \ker P_{\lw}=\ker \Pi \lw.$ This together with~\eqref{7-13} and~\eqref{7-14} implies that  $x\in \mathcal Y_C^\Pi$, where
  \[
  C \lw:=\frac{\lambda \lw }{D\lw (\eta^{\lambda \lw +1}+ \eta ) }, \quad \lw \in \TWW.
  \]
  Note that $C$ is $\Theta$-invariant as $D$ and $\lambda$ are $\Theta$-invariant. Next, 
  \[
  \begin{split}
  &x(\ell, \w, n+1)-\Phi(1, n+\ell, \theta^n \w)x(\ell, \w, n)  \\
&=\sum_{k=0}^{n+1}\varphi_{\ell+k}^{-1}\Phi(n+1-k,\ell+k,\theta^k\w)P_{\Theta^k{\lw}}y(\ell, \w, k) \\
&\phantom{=}-\sum_{k=0}^n \varphi_{\ell+k}^{-1}\Phi(n+1-k,\ell+k,\theta^k\w)P_{\Theta^k{\lw}}y(\ell, \w, k) \\
&\phantom{=}-\sum_{j=2}^{\infty} \varphi_{\ell+n+j}^{-1}\Phi(-j+1,\ell+n+j,\theta^{n+j}\w)Q_{\Theta^{n+j}{\lw}}y(\ell, \w, n+j)  \\
&\phantom{=}+\sum_{j=1}^{\infty} \varphi_{\ell+n+j}^{-1}\Phi(-j+1,\ell+n+j,\theta^{n+j}\w)Q_{\Theta^{n+j}{\lw}}y(\ell, \w, n+j)\\
&=\varphi_{\ell+n+1}^{-1}P_{\Theta^{n+1}{\lw}}y(\ell, \w, n+1)+\varphi_{\ell+n+1}^{-1}Q_{\Theta^{n+1}{\lw}}y(\ell, \w, n+1)\\
&=\varphi_{\ell+n+1}y(\ell, \w, n+1)
  \end{split}
  \]
  for $\lw \in \TWW$, which yields~\eqref{ADM}.

  It remains to establish the uniqueness of $x$. For this it is sufficient to show that if $x\in \mathcal Y_C^\Pi$ is such that~\eqref{ADM} holds for $y\equiv0$, that is if we have
  \[
  x(\ell, \w, n+1)=\Phi(1, n+\ell, \theta^n \w)x(\ell, \w, n) \quad \text{for $(\ell, \w, n)\in \Z^+\times \tilde \Omega \times \Z^+ $},
  \]
  then  $x\equiv 0$.
  Take an arbitrary $\lw \in \TWW$. Then,
  \[
  x(\ell, \w, n)=\Phi(n, \ell, \w )x(\ell, \w, 0)
  \]
  for $n\in \N$. Since $x(\ell, \w, 0)\in \ker P_{\lw}$,  we have that 
  \[
  \begin{split}
  \|x(\ell, \w, 0)\|_{{\lw}} &=\|\Phi(-k,\ell+k,\theta^k\w)Q_{\Theta^k(\ell,\w)}x(\ell, \w, k)\|_{{\lw}}\\
  &\leq D(\ell,\w) \left( \dfrac{\mu_{\ell+k}}
			{\mu_\ell}\right)^{-\lambda\lw}\|x(\ell, \w, k)\|_{\Theta^k{\lw}} \\
   &\le \frac{D\lw }{C\lw }\left( \dfrac{\mu_{\ell+k}}
			{\mu_\ell}\right)^{-\lambda\lw}\|x\|_{\mathcal Y_C^\Pi(\tilde \Omega)}.
  \end{split}
  \]
  Letting $k\to \infty$ yields that $x(\ell, \w, 0)=0$ and consequently $x(\ell, \w, n)=0$ for each $n\in \Z^+$. We conclude that $x\equiv 0$.
  \end{proof}

  We now establish a partial converse to Theorem~\ref{T51}.
  
  \begin{theorem}~\label{T52}
  Let $\mu$ be a growth rate such that for each $n\in\Z^+$ there exists an integer $q_n\geq n+1$ with the property that there are 
		 constants $L_{1},L_{2}>1$ such that
		\begin{equation}
			\label{eq:COND-minimal-growth-1}L_{1} \le \frac{\mu_{q_n}}{\mu_n}
			\le L_{2}, \quad n\in \Z^+.
		\end{equation}
Suppose that there is a random norm $\mathcal N$, a $\Theta$-invariant set  $\tilde \Omega \subset \Omega$ of full measure,  a  family of projections $\Pi=(\Pi \lw)_{\lw \in \TWW}$ and a $\Theta$-invariant random variable $C\colon \TWW \to (0, \infty)$ such that, for each $y\in \mathcal Y^0(\tilde \Omega)$, there exists a unique $x\in \mathcal Y_C^\Pi(\tilde \Omega)$ satisfying~\eqref{ADM}. In addition,  assume that there are $\Theta$-forward invariant random variables $\lambda\colon\TWW\to(0,+\infty)$ and $M\colon\TWW\to[1,+\infty)$ such that 
\begin{equation}
			\label{eq:teo:exist-dich:exponential-bound-1}
                \|\Phi(n,\ell,\w)v\|_{\Theta^n\lw}\le
			M(\ell,\w)\left(\frac{\mu_{\ell+n}}{\mu_\ell}\right)^{\lambda\lw}\|v\|_{\lw},
		\end{equation}
		for all $(n,\ell,\w)\in\Z^+\x \Z^+    \x\tilde\W$ and $v\in X$.
Then, $\Phi$ admits a $\mu$-dichotomy with respect to the family of norms $\mathcal N$.
  \end{theorem}
The proof of Theorem~\ref{T52} will be split into several lemmas. 
Let $$\mathcal T\colon\mathcal{Y}^0(\tilde \Omega)\to\mathcal{Y}_C^\Pi(\tilde \Omega)$$ be a linear operator given by $\mathcal Ty=x$ for $y\in \mathcal Y^0(\tilde \Omega)$, where $x$ is the unique element of $\mathcal Y_C^\Pi(\tilde \Omega)$ such that~\eqref{ADM} holds. 
\begin{lemma}
$\mathcal T$ is a bounded linear operator.
\end{lemma}
\begin{proof}
It is sufficient to show that $\mathcal T$ is a closed operator. To this end, suppose that $(y_n)_{n\in \N}$ is a sequence in $\mathcal Y^0(\tilde \Omega)$ such that $y_n \to y$ in $\mathcal Y^0(\tilde \Omega)$ and $x_n:=\mathcal T y_n \to x$ in $\mathcal Y_C^\Pi(\tilde \Omega)$. Note that 
\begin{equation}\label{limits}
\lim_{n\to \infty}y_n (\ell,\w,k)= y(\ell, \w, k) \quad \text{and} \quad \lim_{n\to \infty}x_n (\ell,\w, k)=x (\ell,\w, k),
\end{equation}
for $(\ell,\w, k) \in \Z^+\times \tilde \Omega\times \Z^+$. Since $\mathcal Ty_n=x_n$, we have that 
\begin{equation}\label{ADM2}
x_n(\ell, \w, k+1)=\Phi(1, k+\ell, \theta^k \w)x_n(\ell, \w, k)+\varphi_{\ell+k+1}^{-1}y_n(\ell, \w, k+1),
\end{equation}
for $n\in \N$ and $(\ell, \w, k)\in \Z^+\times \tilde \Omega \times \Z^+$. By passing to the limit when $n\to \infty$ in~\eqref{ADM2} and using~\eqref{limits} we obtain that 
\[
x(\ell, \w, k+1)=\Phi(1, k+\ell, \theta^k \w)x(\ell, \w, k)+\varphi_{\ell+k+1}^{-1}y(\ell, \w, k+1),
\]
for $(\ell, \w, k)\in \Z^+\times \tilde \Omega \times \Z^+$. Hence, $\mathcal Ty=x$ which yields that $\mathcal T$ is closed. 
\end{proof}
  
Take $Z_\w=\ker \Pi(0, \w)$, $\w \in \tilde \Omega$. For each $\lw\in\TWW$, consider the following subspaces of $X$:
	\[
		V_{\lw}=\left\{v \in X: \sup_{n\geq 0}\|\Phi(n,\ell,\w)v \|_{\Theta^n\lw}<\infty\right
		\}\]
        and
        \[
        Z_{\lw}=
        \begin{cases} 
        Z_{\w}&\,\text{ if } \ell=0\\
        \Phi(\ell,0,\theta^{-\ell}\w)Z_{\theta^{-\ell}\w}&\,\text { if } \ell>0.
	\end{cases}          
        \]
  	Given any $\lw \in \TWW$ and $n \in \Z^+$, it follows easily that
	\[
		\Phi(1,\ell,\w) V_{\lw}\subset V_{\Theta\lw }\quad \text{ and }\quad \Phi(1,\ell,\w)Z_{\lw}= Z_{\Theta\lw}.
	\] 
\begin{lemma}
 \label{spp-1}
		For each $(\ell_0, \w_0)\in\TWW$  we have \begin{equation}\label{SPLIT}
  X=V_{(\ell_0, \w_0)}\oplus Z_{(\ell_0, \w_0)}.
  \end{equation}
	\end{lemma}

	\begin{proof}
 Let us start by proving the desired conclusion for $(0,\w_0)\in \TWW$. For $v \in X$, we 
 consider  $x, y\in \Z^+\times \tilde \Omega \times \Z^+\to X$ defined by
 \[
 x(\ell, \w, n)=\begin{cases}
 v & (\ell, \w, n)=(0, \w_0, 0) \\
 0 & \text{otherwise}
 \end{cases}
 \]
 and 
 \[
 y(\ell, \w, n)=\begin{cases}
-\varphi_1 \Phi(1, 0, \w_0)v & (\ell, \w, n)=(0, \w_0, 1) \\
0 & \text{otherwise.}
 \end{cases}
 \]
 Observe that $y\in \mathcal Y^0(\tilde \Omega)$,  and 
 \[
 x(\ell, \w, n+1)=\Phi(1, \ell+n, \theta^n \w)x(\ell, \w, n)+\varphi_{\ell+n+1}^{-1}y(\ell, \w, n+1), 
 \]
 for $(\ell, \w, n)\in \Z^+\times \tilde \Omega \times \Z^+$. 
 Set $\tilde x:=\mathcal T y\in \mathcal Y_C^\Pi(\tilde \Omega)$. Then,
 \[
 x(\ell, \w, n+1)-\tilde x(\ell, \w, n+1)=\Phi(1, \ell+n, \theta^n \w)(x(\ell, \w, n)-\tilde x(\ell, \w, n)),
 \]
 for $(\ell, \w, n)\in \Z^+\times \tilde \Omega \times \Z^+$. Consequently, 
\[
x(\ell, \w, n)-\tilde x(\ell, \w, n)=\Phi(n, \ell, \w)(x(\ell, \w, 0)-\tilde x(\ell, \w, 0)),
\]
for $(\ell, \w, n)\in \Z^+\times \tilde \Omega \times \Z^+$. In particular, 
\[
x(0, \w_0, n)-\tilde x (0, \w_0, n)=\Phi(n, 0, \w_0)(x(0, \w_0, 0)-\tilde x(0, \w_0, 0)).
\]
Since $\tilde x\in \mathcal Y_C^\Pi(\tilde \Omega)$, we have $\tilde x(0,\w_0,0)\in\ker\Pi_{(0,\w_0)}=Z_{(0,\w_0)}$ and
\[x(0, \w_0, 0)-\tilde x(0, \w_0, 0)=v-\tilde x(0, \w_0, 0)\in V_{(0, \w_0)},
\] 
which yields that 
\[
v=v-\tilde x(0, \w_0, 0)+\tilde x(0, \w_0, 0)\in V_{(0, \w_0)}+Z_{(0, \w_0)}.
\]

Consider now the case $\ell_0\geq 1$. For $v\in X$ we define $\bar y\colon \Z^+\times \tilde \Omega \times \Z^+\to X$ by 
\[
\bar y(\ell, \w, n)=\begin{cases}
\varphi_{\ell_0} v & (\ell, \w, n)=(0, \theta^{-\ell_0} \w_0, \ell_0) \\
0 & \text{otherwise.}
\end{cases}
\]
Then, $\bar y\in \mathcal Y^0(\tilde \Omega)$. Let $\bar x=\mathcal T \bar y\in \mathcal Y_C^\Pi(\tilde \Omega)$. Note that $\bar x(0, \theta^{-\ell_0}\w_0, 0)\in Z_{ \theta^{-\ell_0}\w_0}$. Moreover, in one hand we have 
\[
\bar x(0, \theta^{-\ell_0} \w_0, \ell_0)-\Phi(1, \ell_0-1, \theta^{-1}\w_0)\bar x(0, \theta^{-\ell_0} \w_0, \ell_0-1)=v
\]
with 
\[
\Phi(1, \ell_0-1, \theta^{-1}\w_0)\bar x(0, \theta^{-\ell_0} \w_0, \ell_0-1)=\Phi(\ell_0, 0, \theta^{-\ell_0}\w_0)\bar x(0, \theta^{-\ell_0}\w_0, 0)\in Z_{(\w_0, \ell_0)}.
\]
On the other hand, since $\bar x\in \mathcal Y_C^\Pi(\tilde \Omega)$
\[
\sup_{n\ge 0}\|\Phi(n, \ell_0, \w_0)\bar x(0, \theta^{-\ell_0}\w_0, \ell_0)\|_{\Theta^n (\ell_0, \w_0)}=\sup_{n\ge 0}\|\bar x(0, \theta^{-\ell}\w_0, n+\ell_0)\|_{\Theta^n (\ell_0, \w_0)}<+\infty,
\]
which yields that $\bar x(0, \theta^{-\ell_0}\w_0, \ell_0)\in V_{(\ell_0, \w_0)}$.
 We conclude that 
\[
v=\bar x(0, \theta^{-\ell_0} \w_0, \ell_0)-\Phi(1, \ell_0-1, \theta^{\ell_0-1}\w_0)\bar x(0, \theta^{-\ell_0} \w_0, \ell_0-1)\in V_{(\ell_0, \w_0)}+Z_{(\ell_0, \w_0)}.
\]

        Consider now any $\ell_0\in \Z^+$ and $v \in V_{(\ell_0, \w_0)}\cap Z_{(\ell_0, \w_0)}$. Since $v
		\in Z_{(\ell_0, \w_0)}=\Phi(\ell_0,0,\theta^{-\ell_0}\w_0)Z_{\theta^{-\ell_0}\w_0}$, there is $z \in Z_{\theta^{-\ell_0}\w_0}$ such that
$v=\Phi(\ell_0,0,\theta^{-\ell_0}\w_0) z$. Define $\underline{x}\colon \Z^+\times \tilde \Omega \times \Z^+\to X$ by
\[
\underline{x}(\ell, \w, n)=\begin{cases}
\Phi(n, 0, \theta^{-\ell_0}\w_0)z & (\ell, \w)=(0, \theta^{-\ell_0}\w_0) \\
0 &\text{otherwise.}
\end{cases}
\]
Then, $\underline{x}\in \mathcal Y_C^\Pi(\tilde \Omega)$ and $\underline{x}=\mathcal T0$. Therefore, $\underline{x} = 0$, implying $v = \underline{x}(0, \theta^{-\ell_0} \w_0, \ell_0) = 0$. The proof of the lemma is completed.
\end{proof}
For $\lw \in \TWW$,
let $P_{\lw}\colon X \to V_{\lw}$ be the projection associated to the splitting~\eqref{SPLIT}.

\begin{lemma}\label{measurability}
For each $v\in X$, the map $\lw \mapsto P_{\lw}v$ is measurable.
\end{lemma}

\begin{proof}
Using the same notation as in the proof of Lemma~\ref{spp-1} we have that 
\[
P_{(\ell_0, \w_0)}v=\begin{cases}
x(0, \w_0, 0)-\tilde x(0, \w_0, 0) & \ell_0=0 \\
\bar x(0, \theta^{-\ell_0}\w_0, \ell_0) & \ell_0>0,
\end{cases}
\]
which readily implies the desired conclusion.

\end{proof}

\begin{lemma}
		For each $\lw\in\TWW$ we have that
		\[
			\Phi(1,\ell,\w)\vert_{Z_{\lw}}\colon Z_{\lw}\to Z_{(\ell+1,\theta\w)}
		\]
		is an invertible linear map.
	\end{lemma}

	\begin{proof}
		Let $v \in Z_{(\ell+1,\theta\w)}$. By definition, there exists $z \in Z_{\theta^{-\ell}\w}$ such that
		$$v=\Phi(\ell+1,0,\theta^{-\ell}\w)z=\Phi(1,\ell,\w)\Phi(\ell,0,\theta^{-\ell}\w)z.$$ 
  Since $\Phi(\ell,0,\theta^{-\ell}\w)z\in Z_{\lw}$, we have that $\Phi(1,\ell,\w)\vert_{Z_{\lw}}$
		is surjective.

		Assume now that there exists $v \in Z_{\lw}$ such that $\Phi(1,\ell,\w)
		v=0$. Then, $\Phi(n,\ell,\w) v=0$ for all $n\in \mathbb N$. Therefore, we  have $v \in V_{\lw}$. Thus $v\in V_{\lw}\cap Z_{\lw}$, which together with Lemma~\ref{spp-1} implies $v=0$.
	\end{proof}
Observe that it follows directly from the previous lemma that for each $n\in \Z^+$ and $\lw \in \TWW$, $\Phi(n, \ell, \w)\rvert_{Z_{\lw}}\colon Z_{\lw} \to Z_{\Theta^n \lw}$ is invertible.

\begin{lemma}
For $n\in \Z^+$ and $v\in X$, the map
\[
\lw \mapsto \Phi(-n, \ell+n, \theta^n \w)Q_{\Theta^n \lw}v
\]
is measurable, where $\Phi(-n, \ell+n, \theta^n \w)\colon Z_{\Theta^n \lw}\to Z_{\lw}$ denotes the inverse of $\Phi(n, \ell, \w)\rvert_{Z_{\lw}}$ and $Q_{\lw}=\Id_X-P_{\lw}$.

\end{lemma}

\begin{proof}
Without any loss of generality we may assume that $n\ge 1$. Let us fix an arbitrary $(\ell_0, \w_0)\in \TWW$. We define $y\colon \Z^+\times \tilde \Omega \times \Z^+ \to X$ by 
\[
y(\ell, \w, m)=\begin{cases}
\varphi_{n+\ell_0}\Phi(n, \ell_0, \w_0)Q_{ (\ell_0, \w_0)}v & (\ell, \w, m)=(\ell_0, \w_0, n) \\
0 & (\ell, \w, m)\neq (\ell_0, \w_0, n).
\end{cases}
\]
By Lemma~\ref{measurability} we have that $y$ is measurable. Moreover, $y\in \mathcal  Y^0(\tilde \Omega)$. Let $x:=\mathcal Ty\in \mathcal Y_C^\Pi(\tilde \Omega)$. Then, 
\[
\Phi(-n, \ell_0+n, \theta^n \w_0)Q_{\Theta^n (\ell_0, \w_0)}v=x(\ell_0, \w_0, 0),
\]
which implies the desired claim.

\end{proof}
 
 \begin{lemma}
\label{lemma:estimate-cA-X-1} There are $\Theta$-forward invariant random variables $$D_{1},a\colon\Z^+\!\times\tilde\W\to(0,+\infty)$$ such that
		\[
			\|\Phi (n, \ell, \w)v\|_{(\ell+n, \theta^n \w)}\le D_{1}\lw\left(\frac{\mu_{\ell+n}}{\mu_{\ell}}
			\right)^{-a\lw}\|v\|_{(\ell, \w)},
		\]
		for all $(n,\ell,\w)\in\Z^+\times\Z^+\times\tilde\W$ and $v\in V_{\lw}$.
	\end{lemma}

\begin{proof}[Proof of the lemma]
	Fix $(\ell_0, \w_0)\in\TWW$, $n\in\Z^+$ and $v\in V_{(\ell_0, \w_0)}$.  We define $y\colon \Z^+\times \tilde \Omega \times \Z^+\to X$ by 
 \[
 y(\ell, \w, k)=
\dfrac{\Phi(k-\ell_0, \ell_0, \w_0)v}{\|\Phi(k-\ell_0, \ell_0, \w_0)v\|_{\Theta^{k-\ell_0}(\ell_0, \w_0)}},
 \]
 if $(\ell, \w)=(0, \theta^{-\ell_0}\w_0)$, $\ell_0+1\le k\le \ell_0+n$, and $y(\ell, \w, k)=0$ otherwise. Then, $y\in \mathcal Y^0(\tilde \Omega)$. Moreover, set 
 \[
			x(0, \theta^{-\ell_0}\w_0, k)=
			\begin{cases}
				0,        
    & \quad \text{if}\  k\leq \ell_0    \\[2mm]
				\displaystyle \sum_{j=\ell_0+1}^{k} \frac{\Phi(k-\ell_0, \ell_0, \w_0) v}{\varphi_j\|\Phi(j-\ell_0, \ell_0, \w_0)v\|_{\Theta^{j-\ell_0}(\ell_0, \w_0)}}, & \quad \text{if}\ \ell_0< k\le \ell_0+n \\[2mm]
				\displaystyle \sum_{j=\ell_0+1}^{n} \frac{\Phi(k-\ell_0, \ell_0,\w_0) v}{\varphi_j \|\Phi(j-\ell_0, \ell_0, \w_0) v\|_{\Theta^{j-\ell_0}(\ell_0, \w_0)}}, & \quad \text{if}\ k> \ell_0+n,
			\end{cases}
		\]
  and $x(\ell, \w, k)=0$ for $(\ell, \w)\neq (0, \theta^{-\ell_0}\w_0)$. Since $v\in V_{(\ell_0, \w_0)}$, we have that $x\in \mathcal Y_C^\Pi(\tilde \Omega)$. Moreover, a straightforward computation shows that $\mathcal T y=x$.
As $\|y\|_{\mathcal Y(\tilde \Omega)}=1$, we get that 
\begin{equation}
	\label{eq:TmuQ1-1}
   \begin{split}
   \|\mathcal T\| &\ge \|x\|_{\mathcal Y_C^\Pi(\tilde \Omega)} \\
   &\ge C(\ell_0, \w_0) \|x(0, \theta^{-\ell_0}\w_0, \ell_0+n)
		\|_{\Theta^n(\ell_0,\w_0)} \\
   &=C(\ell_0, \w_0)\|\Phi(n, \ell_0, \w_0)
	 v\|_{\Theta^n(\ell_0,\w_0)} \sum_{j=\ell_0+1}^{\ell_0+n} \frac{1}{\varphi_j\|\Phi(j-\ell_0, \ell_0, \w_0)v\|_{\Theta^{j-\ell_0}(\ell_0, \w_0)}}.
  \end{split}
  \end{equation}
    By~\eqref{eq:teo:exist-dich:exponential-bound-1} and~\eqref{eq:TmuQ1-1} we obtain that 
  	\begin{equation*}
   \begin{split}\|\mathcal T\|& \ge C(\ell_0, \w_0) \|\Phi(n, \ell_0, \w_0)
			v\|_{\Theta^n(\ell_0,\w_0)} \sum_{j=\ell_0+1}^{\ell_0+n} \frac{\mu_{\ell_0}^{\lambda(\ell_0, \w_0)}}{\varphi_j M(\ell_0, \w_0)\mu_j^{\lambda(\ell_0, \w_0)}\|v\|_{(\ell_0, \w_0)}}\\
		&	= \frac{C(\ell_0, \w_0)\|\Phi(n, \ell_0, \w_0) v\|_{\Theta^n(\ell_0,\w_0)}}{M(\ell_0, \w_0)\|v\|_{(\ell_0, \w_0)}}\mu_{\ell_0}^{\lambda(\ell_0, \w_0)}\sum_{j=\ell_0+1}^{\ell_0+n} \frac{1}{\varphi_j\mu_j^{\lambda(\ell_0, \w_0)}}.
   \end{split}
		\end{equation*}
  From Lemma~\ref{prop:ESTIMATES} we have
		\[
\sum_{j=\ell_0+1}^{\ell_0+n}\frac{1}{\varphi_j\mu_j^{\lambda(\ell_0, \w_0)}}=\sum_{j=\ell_0+1}^{\ell_0+n}\mu_j^{-(1+\lambda(\ell_0, \w_0))}\mu_j'\geq 
		\frac{1}{\lambda(\ell_0, \w_0)}\left(\mu_{\ell_0+1}^{-\lambda(\ell_0, \w_0)}-\mu_{\ell_0+n+1}^{-\lambda(\ell_0, \w_0)}\right).
		\]
   Therefore, by~\eqref{eq:COND-minimal-growth-1}, if $n \ge q_{\ell_0+1}-\ell_0-1$,
		\[
			\begin{split}
				\|\mathcal T\|&\ge \frac{C(\ell_0, \w_0)\|\Phi(n, \ell_0, \w_0)v\|_{\Theta^n(\ell_0,\w_0)}}{M(\ell_0, \w_0)\lambda (\ell_0, \w_0)\|v\|_{(\ell_0,\w_0)}}\mu_{\ell_0}^{\lambda(\ell_0, \w_0)}
    \left (\mu_{\ell_0+1}^{-\lambda(\ell_0, \w_0)}-\mu_{q_{\ell_0+1}}^{-\lambda(\ell_0, \w_0)}
    \right )\\
    &=\frac{C(\ell_0, \w_0)\|\Phi(n, \ell_0, \w_0)v\|_{\Theta^n(\ell_0,\w_0)}}{M(\ell_0, \w_0)\lambda (\ell_0, \w_0)\|v\|_{(\ell_0,\w_0)}}\frac{\mu_{\ell_0}^{\lambda(\ell_0, \w_0)}}{\mu_{\ell_0+1}^{\lambda(\ell_0, \w_0)} }\left (1-\frac{\mu_{\ell_0+1}^{\lambda(\ell_0, \w_0)}}{\mu_{q_{\ell_0+1}}^{\lambda(\ell_0, \w_0)}} \right )\\
    &\ge \frac{C(\ell_0, \w_0)\|\Phi(n, \ell_0, \w_0)v\|_{\Theta^n(\ell_0,\w_0)}}{M(\ell_0, \w_0)\lambda (\ell_0, \w_0)\|v\|_{(\ell_0,\w_0)}}\frac{\mu_{\ell_0}^{\lambda(\ell_0, \w_0)}}{\mu_{q_{\ell_0}}^{\lambda (\ell_0, \w_0)}}\left (1-\frac{\mu_{\ell_0+1}^{\lambda(\ell_0, \w_0)}}{\mu_{q_{\ell_0+1}}^{\lambda(\ell_0, \w_0)}} \right )\\
    &\ge \frac{C(\ell_0, \w_0)\|\Phi(n, \ell_0, \w_0)v\|_{\Theta^n(\ell_0,\w_0)}}{M(\ell_0, \w_0)\lambda (\ell_0, \w_0)\|v\|_{(\ell_0,\w_0)}}L_2^{-\lambda(\ell_0, \w_0)}(1-L_1^{-\lambda(\ell_0, \w_0)}).
			\end{split}
		\]
  We conclude that, if $n \ge q_{\ell_0+1}-\ell_0-1$, we have
		\[
  \begin{split}
&\|\Phi(n, \ell_0, \w_0)v\|_{\Theta^n (\ell_0, \w_0)} \\
&\le \frac{1}{C(\ell_0, \w_0)}\lambda(\ell_0, \w_0) M (\ell_0, \w_0)L_{2}^{\lambda(\ell_0, \w_0)} L_{1}^{\lambda(\ell_0, \w_0)}(L_{1}^{\lambda(\ell_0, \w_0)}
			-1)^{-1}\|\mathcal T\| \cdot \|v\|_{(\ell_0, \w_0)}.
   \end{split}
		\]
   On the other hand, using~\eqref{eq:teo:exist-dich:exponential-bound-1}, we have, if $n < q_{\ell_0+1}-\ell_0-1$,
  \[
  \begin{split}
  \|\Phi(n, \ell_0, \w_0)v\|_{\Theta^n(\ell_0, \w_0)} &\le M(\ell_0, \w_0) \left (\frac{\mu_{\ell_0+n}}{\mu_{\ell_0}} \right )^{\lambda (\ell_0, \w_0)}\|v\|_{(\ell_0, \w_0)}\\
  &\le M(\ell_0, \w_0)\left (\frac{\mu_{q_{\ell_0+1}}}{\mu_{\ell_0}}\right )^{\lambda (\ell_0, \w_0)}\|v\|_{(\ell_0, \w_0)}\\
  &=M(\ell_0, \w_0)\left (\frac{\mu_{q_{\ell_0+1}}}{\mu_{\ell_0+1}}\right )^{\lambda (\ell_0, \w_0)} \left (\frac{\mu_{\ell_0+1}}{\mu_{\ell_0}}\right )^{\lambda (\ell_0, \w_0)}
  \|v\|_{(\ell_0, \w_0)}\\
  &\le M(\ell_0, \w_0)L_2^{\lambda (\ell_0, \w_0)}\eta^{\lambda (\ell_0, \w_0)}  \|v\|_{(\ell_0, \w_0)}.
  \end{split}
  \]
  We conclude that, for any $n \ge 0$, we have
		\begin{equation}\label{eq:M2-1}
  \|\Phi(n, \ell_0, \w_0)v\|_{\Theta^n(\ell_0,\w_0)} \le M_{2}(\ell_0, \w_0)\|v\|_{(\ell_0, \w_0)},
            \end{equation}
            where $M_2\colon \TWW \to (0, \infty)$ is a $\Theta$-invariant random variable. Taking into account~\eqref{prop:ESTIMATES-est-=1}, for $n > 0$ we have
		\begin{equation}
			\label{eq:bound-importamt-sum-1}
			\begin{split}
				\sum_{j=\ell_0+1}^{\ell_0+n} \frac{1}{\varphi_j}\ge \log \frac{\mu_{\ell_0+n+1}}{\mu_{\ell_0+1}}
				.
			\end{split}
		\end{equation}
By~\eqref{eq:TmuQ1-1}, \eqref{eq:M2-1} and~\eqref{eq:bound-importamt-sum-1},
		we have that 
		\[
			\begin{split}
			\|\mathcal T\| &\ge C(\ell_0, \w_0)\|\Phi(n, \ell_0, \w_0)v\|_{\Theta^n(\ell_0,\w_0)}\sum_{j=\ell_0+1}^{\ell_0+n}\frac{1}{\varphi_j\|\Phi(j-\ell_0, \ell_0, \w_0)v\|_{(j, \theta^{j-\ell_0}\w_0)}}\\
				&\ge \frac{C(\ell_0, \w_0)\|\Phi(n, \ell_0, \w_0)v\|_{\Theta^n(\ell_0,\w_0)}}{M_{2}(\ell_0, \w_0) \|v\|_{(\ell_0, \w_0)}}\log \frac{\mu_{\ell_0+n+1}}{\mu_{\ell_0+1}}
				\\&\ge \frac{C(\ell_0, \w_0)\|\Phi(n,\ell_0, \w_0)v\|_{\Theta^n(\ell_0,\w_0)}}{M_{2}(\ell_0, \w_0) \|v\|_{(\ell_0, \w_0)}}\log \frac{\mu_{\ell_0+n+1}}{\eta \mu_{\ell_0}}
				.\\
			\end{split}
		\]
 Taking into account that
		\[
			\|\mathcal T\| \frac{M_2(\ell_0, \w_0)}{C(\ell_0, \w_0)} \left(\log \frac{\mu_{\ell_0+n+1}}{\eta \mu_{\ell_0}}\right
			)^{-1}\le \e^{-1}\quad 
   \iff \quad \mu_{\ell_0+n+1} \ge N_{0}(\ell_0,\w_0) \mu_{\ell_0},
		\]
		where
		\[
			N_0=N_0(\ell_0, \w_0):=\eta \exp \left \{\e\|\mathcal T\| \frac{M_2(\ell_0, \w_0)}{C(\ell_0, \w_0)} \right \},
		\]
		we conclude that
		\begin{equation}
\label{eq:crucial-1}\|\Phi(n, \ell_0, \w_0)v\|_{\Theta^n(\ell_0,\w_0)} 
\le \e^{-1}\|v\|
			_{(\ell_0, \w_0)},
		\end{equation} 
  for $(\ell_0,\w_0) \in \TWW$ and $n \in \Z^+$ is such that $\mu_{\ell_0+n+1} \ge N_{0}(\ell_0, \w_0)\mu_{\ell_0}$. We also observe that $N_0$ is $\Theta$-forward invariant (since $M_2$ and $C$ are such).

 Take $\lw \in \TWW$. Define $\gamma(\ell,0)=\ell$ and $\gamma(\ell,k)=q_{\gamma(\ell,k-1)}$ for $k\in \N$. Let $K_0=K_0(\ell,\w)$ be the smallest positive integer such that $L_{1}^{K_0(\ell, \w)}\geq N_0(\ell, \w)$. We  observe that $K_0$ is $\Theta$-forward invariant.
		Due to~\eqref{eq:COND-minimal-growth-1}, we have for all $j \in \N$,
		\begin{equation*}
			\begin{split}
				\frac{\mu_{\gamma(\ell,(j+1)K_0(\ell, \w))}}{\mu_{\gamma(\ell,jK_0(\ell, \w))}}= \prod_{k=1}^{K_0(\ell, \w)} \frac{\mu_{\gamma(\ell,jK_0(\ell, \w)+k)}}{\mu_{\gamma(\ell,jK_0(\ell, \w)+k-1)}}\ge L_{1}^{K_0(\ell, \w)}\geq N_0(\ell, \w).
			\end{split}
	\end{equation*}
 For any $n\in\Z^+$ consider the largest $r=r(\ell, \w) \in \Z^+$ with the property that   $\ell+n\ge \gamma(\ell,r K_0)$. Note that $r$ is $\Theta$-forward invariant.
 For $j=0,\ldots,r-1$, set $\hat\gamma(\ell,j)=\gamma(\ell,(j+1)K_0)-\gamma(\ell,jK_0)$. Clearly $$\sum_{j=1}^{r-1}\hat\gamma(\ell,j)=\gamma(\ell,rK_0)-\gamma(\ell,K_0).$$
 Moreover, for all $j=0,\ldots,r-1$ and $u\in X$ we have
   \begin{equation}
		\label{eq:bound-cA-M2-1}
        \begin{split}  &\|\Phi(\hat\gamma(\ell,j),\gamma(\ell,jK_0),\theta^{\gamma(\ell,jK_0)-\ell}\w)u\|_{\Theta^{\gamma(\ell,(j+1)K_0)-\ell}\lw}\\ &
   \qquad\leq  \e^{-1} 
   \|u \|_{\Theta^{\gamma(\ell,jK_0)-\ell}\lw}.
   \end{split}
   \end{equation}
  Using~\eqref{eq:M2-1}, \eqref{eq:crucial-1} and~\eqref{eq:bound-cA-M2-1}
    \begin{equation}
			\label{eq:aux-bound-cAmn-1}
    \begin{split}
    & \|\Phi(n,\ell,\w)v\|_{\Theta^n\lw}\\
    & \,
    = \|\Phi(n+\ell-\gamma(\ell,rK_0),\gamma(\ell,rK_0),\theta^{\gamma(\ell,rK_0)-\ell}\w)\Phi(\gamma(\ell,rK_0)-\ell,\ell,\w)v\|_{\Theta^n\lw}\\
&\le M_2\lw \|\Phi(\gamma(\ell, rK_0)-\ell, \ell, \w)v\|_{\Theta^{\gamma(\ell, rK_0)-\ell}\lw}\\
    &  \,
    \leq M_2 \lw 
    \|
    \Phi(\hat\gamma(\ell,r-1),\gamma(\ell,(r-1)K_0),\theta^{\gamma(\ell,(r-1)K_0)-\ell}\w)\cdots\\
    &\quad\cdots\Phi(\hat\gamma(\ell, 1),\gamma(\ell,K_0),\theta^{\gamma(\ell,K_0)-\ell}\w)\Phi(\hat \gamma(\ell, 0), \ell,\w)v\|_{\Theta^{\gamma(\ell, rK_0)-\ell}\lw}
    \\
    &
    \,\le M_2 \lw  \e^{-r} 
				\|v\|_{\lw}.\\
			\end{split}
   \end{equation}
Noting that
		\begin{equation*}
			\begin{split}
				\dfrac{\mu_{\ell+n}}{\mu_\ell}\leq\frac{\mu_
    {\gamma(\ell,(r+1)K_0)}}{\mu_\ell}=\prod_{j=1}^{r+1} \frac{\mu_{\gamma(\ell,j K_0)}}{\mu_{\gamma(\ell,(j-1) K_0)}}\le
			L_{2}^{(r+1)K_0},
			\end{split}
			\end{equation*}		
		we get
		\begin{equation}
			\label{eq:aux-bound-cAmn-2-1}\e^{-r-1}\leq  \left(\dfrac{\mu_{\ell+n}}{\mu_\ell}\right
			)^{-1/(K_0 \log L_2)}.
		\end{equation}
		It follows from~\eqref{eq:aux-bound-cAmn-1} and~\eqref{eq:aux-bound-cAmn-2-1}
		that, for $n \in\Z^+$,
		\[
			\begin{split}
				\|\Phi(n,\ell,\w)v\|_{\Theta^n\lw}&
        \le \e M_2 \lw   \left(\dfrac{\mu_{\ell+n}}{\mu_\ell}\right
				)^{-1/(K_0\lw \log L_2)}\|v\|_{\lw},
			\end{split}
		\]
		and the statement holds with	 $D_1(\ell,\w)=\e  M_2 \lw$ and $a\lw=(K_{0}\lw \log L_{2})^{-1}$.
   
\end{proof}
\begin{lemma}
		\label{lemma:estimate-cA-Z-1} There are $\Theta$-forward invariant random variables $$D_{2},b\colon\Z^+\!\times\tilde\W\to(0,+\infty)$$ such that 
				\[
		\|\Phi(-n,\ell+n,\theta^{n}\w)v\|_{\lw}
		\le D_{2}\lw\left( \dfrac{\mu_{\ell+n}}{\mu_\ell} \right)^{-b\lw}\|v\|_{\Theta^{n}\lw},\]
				for all $\lw\in\TWW$, $n\in\Z^+$ and $v\in Z_{\Theta^n\lw}$.

	\end{lemma}
\begin{proof}[Proof of the lemma]
Take $(\ell_0, \w_0)\in \TWW$ and 
$s\in  \Z^+$.	Let $z \in Z_{(0,\theta^{-\ell_0}\w_0)} \setminus\{0\}$. We define $y\colon \Z^+\times \tilde \Omega \times \Z^+\to X$ by 
\[
y(0, \theta^{-\ell_0}\w_0, k)= 
  \begin{cases}
				-\dfrac{\Phi(k, 0, \theta^
 {-\ell_0}\w_0)z}{\|\Phi(k, 0, \theta^
 {-\ell_0}\w_0)z\|_{\Theta^k(0, \theta^{-\ell_0}\w_0)}} & \quad \text{if }1\le  k\le 
   \ell_0+s \\
				0,                                                     & \quad \text{otherwise,}
    \end{cases}
\]
and $y(\ell, \w, k)=0$ for $(\ell, \w)\neq (0, \theta^{-\ell_0}\w_0)$ and $k\in \Z^+$.
Clearly, $y\in \mathcal Y^0(\tilde \Omega)$. In addition, we define $x\colon \Z^+\times \tilde \Omega \times \Z^+\to X$ by 
\[
			x(0, \theta^{-\ell_0}\w_0, k)=
			\begin{cases}
								\displaystyle \sum_{j=k+1}^{\ell_0+s} \frac{\Phi(k,0,\theta^{-\ell_0}\w_0) z}{\varphi_j\|\Phi(j,0,\theta^{-\ell_0}\w_0)z\|_{\Theta^j(0, \theta^{-\ell_0}\w_0)}} & \quad \text{if}\ 0\le k\le \ell_0+s-1 \\[2mm]
				0, & \quad \text{otherwise, }
			\end{cases}
   \]
   and $x(\ell, \w, k)=0$ for $(\ell, \w)\neq (0, \theta^{-\ell_0}\w_0)$ and $k\in \Z^+$. Then, $x\in \mathcal Y_C^\Pi(\tilde \Omega)$ and it is straightforward to verify that $\mathcal Ty=x$. Consequently, since $\|y\|_{\mathcal Y(\tilde \Omega)}=1$, we get that 
   \[
   \begin{split}\|\mathcal T\| &\ge \|x\|_{\mathcal Y_C^\Pi(\tilde \Omega)} \\
   &\geq C(\ell_0, \w_0)\|\Phi(k, 0, \theta^{-\ell_0}\w_0)
	 z\|_{\Theta^k(0, \theta^{-\ell_0}\w_0)} \sum_{j=k+1}^{\ell_0+s} \frac{1}{\varphi_j\|\Phi(j, 0, \theta^{-\ell_0}\w_0)z\|_{\Theta^j(0, \theta^{-\ell_0}\w_0)}},
  \end{split}
		\]
  for $0\le k\le \ell_0+s-1$. Letting $s\to \infty$ we obtain that for $k \in \Z^+$ and $z \in Z_{(0,\theta^{-\ell_0}\w_0)}\setminus
		\{0\}$,
		\[
		\|\mathcal T\| \ge C(\ell_0, \w_0) \|\Phi(k, 0, \theta^{-\ell_0}\w_0)
 z\|_{\Theta^k(0, \theta^{-\ell_0}\w_0)} \sum_{j=k+1}^{\infty} \frac{1}{\varphi_j\|\Phi(j, 0, \theta^{-\ell_0}\w_0)z\|_{\Theta^j(0, \theta^{-\ell_0}\w_0)}}
			.\]
Thus, 
we have by~\eqref{prop:ESTIMATES-est-le-<0ne1}, ~\eqref{eq:COND-minimal-growth-1}
		and~\eqref{eq:teo:exist-dich:exponential-bound-1} that 
\[
\begin{split}
  & \frac{1}{\|\Phi(\ell_0,0,\theta^{-\ell_0}\w_0)z\|_{(\ell_0, \w_0)}} \\
  &\ge
   \frac{C(\ell_0, \w_0)}{\|\mathcal T\|}\sum_{j=\ell_0+n+1}^{q_{\ell_0+n+1}-1}\frac{1}{\varphi_j \|\Phi(j, 0, \theta^{-\ell_0}\w_0)z\|_{\Theta^j(0, \theta^{-\ell_0}\w_0)}}\\
   &= \frac{C(\ell_0, \w_0)}{\|\mathcal T\|}\sum_{j=\ell_0+n+1}^{q_{\ell_0+n+1}-1}\frac{1}{\varphi_j \|\Phi(j-n-\ell_0, n+\ell_0, \theta^n\w_0)\Phi(n+\ell_0,0,\theta^{-\ell_0}\w_0)z\|_{\Theta^j (0, \theta^{-\ell_0}\w_0)}}\\
   &\ge \frac{C(\ell_0, \w_0)}{M(\ell_0, \w_0)\|\mathcal T\|}\sum_{j=\ell_0+n+1}^{q_{\ell_0+n+1}-1}\frac{1}{\varphi_j (\mu_j/\mu_{n+\ell_0})^{\lambda} \|\Phi(n+\ell_0,0,\theta^{-\ell_0}\w_0)z\|_{\Theta^n(\ell_0,\w_0)}}\\
   &=\frac{C(\ell_0, \w_0)\mu_{n+\ell_0}^{\lambda}}{M(\ell_0, \w_0) \|\mathcal T\| \cdot \|\Phi(n+\ell_0,0, \theta^{-\ell_0}\w_0)z\|_{\Theta^n(\ell_0,\w_0)}}  \sum_{j=\ell_0+n+1}^{q_{\ell_0+n+1}-1}\mu_j^{-\lambda-1}\mu_j'   \\
   &\ge \frac{C(\ell_0, \w_0)\mu_{n+\ell_0}^{\lambda}}{M(\ell_0, \w_0) \lambda \|\mathcal T\| \cdot \|\Phi(n+\ell_0,0, \theta^{-\ell_0}\w_0)z\|_{\Theta^n(\ell_0,\w_0)}}(\mu_{\ell_0+n+1}^{-\lambda}-\mu_{q_{\ell_0+n+1}}^{-\lambda})\\
   &=\frac{C(\ell_0, \w_0)}{M (\ell_0, \w_0)\lambda \|\mathcal T\| \cdot \|\Phi(n+\ell_0,0,\theta^{-\ell_0}\w_0)z\|_{\Theta^n(\ell_0,\w_0)}}\left (\frac{\mu_{n+\ell_0}}{\mu_{\ell_0+n+1}}\right )^{\lambda} \left (1-\left (\frac{\mu_{\ell_0+n+1}}{\mu_{q_{\ell_0+n+1}}}\right )^{\lambda}\right ) \\
   &\ge \frac{C(\ell_0, \w_0)}{M (\ell_0, \w_0)\lambda \|\mathcal T\| \cdot \|\Phi(n+\ell_0,0,\theta^{-\ell_0}\w)z\|_{\Theta^n(\ell_0,\w_0)}}\eta^{-\lambda} (1-L_1^{-\lambda}),
  			\end{split}
			\]	
   for $n\in \Z^+$, where $\lambda=\lambda (\ell_0, \w_0)$. Thus, for any $n\in\Z^+$ we get
		\begin{equation*}
   \|\Phi(\ell_0+n,0,\theta^{-\ell_0}\w_0)z\|_{\Theta^n(\ell_0, \w_0)} \ge L (\ell_0, \w_0)\|\Phi(\ell_0,0,\theta^{-\ell_0}\w_0)z\|_{(\ell_0, \w_0)},			\end{equation*}
		where $L\colon \TWW \to (0, \infty)$ is a $\Theta$-forward invariant random variable.

  Therefore, for all $n\geq1$
\[
\begin{split}
& \frac{1}{\|\Phi(\ell_0,0,\theta^{-\ell_0}\w_0)z\|_{(\ell_0, \w_0)}}\\
 &\ge \frac{C(\ell_0, \w_0)}{\|\mathcal T\|}\sum_{j=\ell_0+1}^{\ell_0+n}\frac{1}{\varphi_j \|\Phi(j,0,\theta^{-\ell_0}\w_0)z\|_{\Theta^j (0, \theta^{-\ell_0}\w_0)}} \\
&\ge \frac{L(\ell_0, \w_0)C(\ell_0, \w_0)}{\|\mathcal T\| \cdot \|\Phi(\ell_0+n,0,\theta^{-\ell_0}\w_0)z\|_{\Theta^n(\ell_0, \w_0)}}\sum_{j=\ell_0+1}^{\ell_0+n} \mu_j^{-1}\mu_j' \\
&\ge \frac{L(\ell_0, \w_0)C(\ell_0, \w_0)}{\|\mathcal T\| \cdot \|\Phi(\ell_0+n,0,\theta^{-\ell_0}\w)z\|_{\Theta^n(\ell_0, \w_0)}} \log \frac{\mu_{\ell_0+n+1}}{\mu_{\ell_0+1}}\\
&\ge \frac{L(\ell_0, \w_0)C(\ell_0, \w_0)}{\|\mathcal T\| \cdot \|\Phi(\ell_0+n,0,\theta^{-\ell_0}\w)z\|_{\Theta^n(\ell_0, \w_0)}} \log \left (\frac{\mu_{\ell_0+n}}{\mu_{\ell_0}}\cdot \frac{\mu_{\ell_0}}{\mu_{\ell_0+1}}\right ) \\
&\ge \frac{L(\ell_0, \w_0)C(\ell_0, \w_0)}{\|\mathcal T\| \cdot \|\Phi(\ell_0+n,0,\theta^{-\ell_0}\w)z\|_{\Theta^n(\ell_0, \w_0)}} \log \left (\frac{\mu_{\ell_0+n}}{\mu_{\ell_0}}\cdot \eta^{-1}\right ).  
\end{split}
 \]
 We conclude that 
\[
\| \Phi(\ell_0+n,0,\theta^{-\ell_0}\w)z\|_{\Theta^n(\ell_0,\w_0)}\ge \e\|\Phi(\ell_0,0,\theta^{-\ell_0}\w)z\|_{(\ell_0, \w_0)},
\]
for $(\ell_0, \w_0) \in \TWW$ and $n\in \N$,
provided that $\mu_{\ell_0+n} \ge N_0\mu_{\ell_0}$, where 
\[
N_0=N_0(\ell_0, \w_0)=\eta \exp \left \{\frac{\e \|\mathcal T\|}{C (\ell_0, \w_0) L(\ell_0, \w_0)} \right \}.
\]
We observe that $N_0\colon \TWW \to (0, \infty)$ is $\Theta$-forward invariant. As in the proof of the previous lemma, let $K_0=K_0 \lw$ be the smallest positive integer such that $L_1^{K_0\lw}\ge N_0\lw$.

Take now arbitrary $\lw \in \TWW$ and $n\in \Z^+$.  Let $\gamma(\ell, 0)=\ell$ and $\gamma(\ell, k)=q_{\gamma(\ell, k-1)}$ for $k\in \N$.
Let $r$ be the largest integer  such that $\ell+n\ge \gamma(\ell, rK_0)$.
Then, by writing $K_0$ instead of $K_0 \lw$ we have that 
\begin{equation}\label{eq:auxx-1}
\begin{split}
&\|\Phi(\ell+n,0,\theta^{-\ell}\w)z\|_{\Theta^n(\ell,\w)}\\
&\ge L \lw \|\Phi(\gamma(\ell, rK_0), 0, \theta^{-\ell}\w)z\|_{(\gamma(\ell, rK_0), \theta^{\gamma(\ell, rK_0)-\ell}\w)} \\
&\ge L\lw \e^ r\|\Phi(\ell,0,\theta^{-\ell}\w)z\|_{\lw}.
\end{split}
\end{equation}
On the other hand, 
\[
\frac{\mu_{\ell+n}}{\mu_\ell}\le \frac{\mu_{\gamma(\ell, (r+1)K_0)}}{\mu_\ell}  \le L_2^{(r+1)K_0 },
\]
which is equivalent to
\[
\e^r\geq \frac1\e\left(\frac{\mu_{\ell+n}}{\mu_\ell}\right)^{1/(K_0   \log L_2)}.
\]
From~\eqref{eq:auxx-1} we conclude that 
\begin{equation}\label{ab-1}
\|\Phi(\ell+n,0,\theta^{-\ell}\w)z\|_{\Theta^n(\ell,\w)}\ge \frac{L\lw} \e  \left(\frac{\mu_{\ell+n}}{\mu_\ell}\right )^{1/(K_0  \log L_2)}\|\Phi(\ell, 0, \theta^{-\ell}\w)z\|_{\lw},
\end{equation}
for $n\in\Z^+$ and $z\in Z_{(0,\theta^{-\ell}\w)}$.

We are now in a position to complete the proof of the lemma. Take $v\in Z(\ell+n, \theta^n \w)$. Then, there exists $z\in Z_{(0,\theta^{-\ell}\w)}$ such that 
\[
v=\Phi(\ell+n, 0, \theta^{-\ell}\w)z.
\]
Note that 
\[
\Phi(-n, \ell+n, \theta^{\ell+n} \w)v=\Phi(-n, \ell+n, \theta^{\ell+n} \w)\Phi(\ell+n, 0, \theta^{-\ell}\w)z=\Phi(\ell, 0, \theta^{-\ell}\w)z.
\]
Hence, \eqref{ab-1} implies that 
\[
\begin{split}
&\|\Phi(-n, \ell+n, \theta^{\ell+n} \w)v\|_{(\ell, \w)} \\
&=\|\Phi(\ell, 0, \theta^{-\ell}\w)z\|_{(\ell, \w)} \\
&\le \frac{e}{L\lw }\left (\frac{\mu_{\ell+n}}{\mu_\ell}\right )^{-1/(K_0 \log L_2)}\|\Phi(\ell+n, 0, \theta^{-\ell}\w)z\|_{\Theta^n \lw }\\
&=\frac{e}{L\lw }\left (\frac{\mu_{\ell+n}}{\mu_\ell}\right )^{-1/(K_0 \log L_2)}\|v\|_{\Theta^n \lw},
\end{split}
\]
which gives the desired conclusion.
\end{proof}

\begin{lemma}
There exists a $\Theta$-forward invariant random variable $\bar C\colon \TWW \to (0, \infty)$ such that 
\[
\|P_{\lw}v\|_{\lw} \le \bar C \lw \|v\|_{\lw}, \quad \lw \in \TWW \ \text{and} \ v\in X.
\]
\end{lemma}

\begin{proof}[Proof of the lemma]
Let $\|P_{\lw}\|_{\lw}:=\sup_{\|v\|_{\lw}\le 1}\|P_{\lw}v\|_{\lw}$. It follows from the proof of Lemma~4.2 in~\cite{Mi.Ra.Sc.1998}
	that
	\begin{equation}
		\label{eq:prelim-bound-Pn-1}\|P_{\lw}\|_{\lw} \le 2/\zeta{\lw},
	\end{equation}
	where
	\[
		\zeta{\lw}=\inf \left\{ \|u+z\|_{\lw}: \ (u,z) \in V_{\lw}\times Z_{\lw}, \ 
		\|u\|_{\lw}=\|z\|_{\lw}=1 \right\}.
	\]
 From~\eqref{eq:prelim-bound-Pn-1} it suffices to prove that there exists a $\Theta$-forward invariant random variable $c\colon \TWW \to (0, \infty)$ such that \begin{equation}\label{gammac}\zeta{\lw}\ge c\lw\quad \text{ for $\lw \in \TWW$.}\end{equation} Let $(u,z) \in V_{\w}\oplus Z_{\w}$ with $\|u\|_\w=\|z\|_\w=1$. Using~\eqref{eq:teo:exist-dich:exponential-bound-1},
		we conclude that
  \[
			\|\Phi(n, \ell, \w)(u+z)\|_{\Theta^n(\ell,\w)} \le M \lw \left(\frac{\mu_{n+\ell}}{\mu_\ell}\right)^{\lambda \lw } \|u+z\|_{\lw}.
		\]
  Thus, using Lemma~\ref{lemma:estimate-cA-X-1} and Lemma~\ref{lemma:estimate-cA-Z-1},
		we obtain, for $n \ge 0$,
		\begin{equation}\label{942-1}
			\begin{split}
				&\|u+z\|_{\lw}\\
    &\ge \frac{1}{M\lw \left(\frac{\mu_{n+\ell}}{\mu_\ell}\right)^{\lambda \lw}}\|\Phi(n, \ell, \w)(u+z)\|_{\Theta^n(\ell,\w)}
				\\
				&\ge \frac{1}{M\lw \left(\frac{\mu_{n+\ell}}{\mu_\ell}\right)^{\lambda \lw}}\left(\|\Phi(n,\ell, \w)z\|_{\Theta^n(\ell,\w)}-\|\Phi(n, \ell, \w)u\|_{\Theta^n(\ell,\w)}\right)\\
				&\ge \frac{1}{M\lw \left(\frac{\mu_{n+\ell}}{\mu_\ell}\right)^{\lambda \lw}}\left(\frac{1}{D_{1}\lw }
				\left(\frac{\mu_{n+\ell}}{\mu_\ell}\right)^{a \lw}-D_{1}\lw \left(\frac{\mu_{n+\ell}}{\mu_\ell}\right)^{-a \lw}\right),
			\end{split}
		\end{equation}
where without any loss of generality we  assumed that $D_1=D_2$ and $a=b$.  

Set $\gamma(\ell, 0):=q_\ell$ and $\gamma(\ell, k):=q_{\gamma(\ell, k-1)}$ for $k\ge 1$.
Let $m\colon \TWW \to \N$ be a forward $\Theta$-invariant random variable such that 
\[
\frac{1}{D_{1}\lw }L_1^{a\lw m\lw}-D_1 \lw L_2^{-a\lw m\lw}>0.
\]
By applying~\eqref{942-1} for $n=\gamma(\ell, m \lw)-\ell$, we obtain that~\eqref{gammac} holds with
\[
c\lw :=\frac{1}{M \lw L_2^{\lambda \lw m\lw }}\left (\frac{1}{D_{1}\lw }L_1^{a\lw m\lw}-D_1 \lw L_2^{-a\lw m\lw}\right ).
\]

\end{proof}

\section{Robustness}\label{se:robustness}
In this section, we discuss the robustness of $\mu$-dichotomies for NRDS.
\begin{theorem}
Let $\mu$ be a growth rate satisfying the assumptions of Theorem~\ref{T52}.
Let $\Phi, \Psi \colon \Z^+\times \Z^+\times \Omega \times X\to X$ be two NRDS such that $\Phi$ admits a $\mu$-dichotomy with respect to a random norm $\mathcal N$. In addition,  assume that there are $\Theta$-forward invariant random variables $\lambda\colon\TWW\to(0,+\infty)$ and $M\colon\TWW\to[1,+\infty)$ such that~\eqref{eq:teo:exist-dich:exponential-bound-1} holds.
Then, there exists a $\Theta$-forward invariant random variable $c\colon \TWW \to (0, \infty)$ such that if 
\begin{equation}\label{rob}
\|(\Phi(1, \ell, \w)-\Psi(1, \ell, \w))v\|_{\Theta \lw}\le \frac{c\lw}{\varphi_{\ell+1}} \|v\|_{\lw} 
\end{equation}
for $\lw \in \TWW$ and $v\in X$, then $\Psi$  admits a $\mu$-dichotomy with respect to the random norm $\mathcal N$.
\end{theorem}
\begin{proof}
Let $C\colon \TWW \to (0, \infty)$ and projections $\Pi\lw$, $\lw \in \TWW$ be given by Theorem~\ref{T51}. Let $c\colon \TWW \to (0, \infty)$ be an $\Theta$-forward random variable such that 
\[
c \lw \le \rho C\lw \quad \text{and} \quad c\lw <\frac{1}{2M\lw},
\]
for $\lw \in \TWW$,
where $\rho>0$ is such that $\rho \|\mathcal T\|<1$ and  $\mathcal T\colon \mathcal Y^0(\tilde \Omega)\to \mathcal Y_C^\Pi(\tilde \Omega)$ is the linear operator introduced in the proof of Theorem~\ref{T52}.

Take an arbitrary $y\in \mathcal Y^0(\tilde \Omega)$. We define $\mathcal G\colon \mathcal Y_C^\Pi(\tilde \Omega)\to \mathcal Y^0(\tilde \Omega)$ by 
\[
(\mathcal G x)(\ell, \w, n):=
\begin{cases}
0 &n=0 \\
\varphi_{\ell+n}\Psi_\Phi(1, n+\ell-1, \theta^{n-1}\w)x(\ell, \w, n-1)+y(\ell, \w, n) & n>0,
\end{cases}
\]
for $(\ell, \w, n)\in \Z^+\times \tilde \Omega \times \Z^+$ where
\[
\Psi_\Phi(n, \ell, \w):=\Psi(n, \ell, \w)-\Phi(n, \ell, \w).
\]
Note that for $x\in \mathcal Y_C^\Pi(\tilde \Omega)$  and $(\ell, \w, n)\in \Z^+\times \tilde \Omega\times \Z^+$ with $n>0$ we have using~\eqref{rob} that 
\[
\|(\mathcal G x)(\ell, \w, n)\|_{\Theta^n \lw}\le \rho C\lw \|x(\ell, \w, n-1)\|_{\Theta^{n-1}\lw}+\|y\|_{\mathcal Y^0(\tilde \Omega)},
\] 
which implies that 
\[
\|\mathcal Gx\|_{\mathcal Y^0(\tilde \Omega)}\le \rho \|x\|_{\mathcal Y_C^\Pi(\tilde \Omega)}+\|y\|_{\mathcal Y^0(\tilde \Omega)}.
\]
In particular, $\mathcal G$ is well-defined. Moreover, the same argument shows that 
\begin{equation}\label{contr1}
   \|\mathcal G x_1-\mathcal G x_2\|_{\mathcal Y^0(\tilde \Omega)}\le \rho \|x_1-x_2\|_{\mathcal Y_C^\Pi(\tilde \Omega)}, \quad \text{for $x_1, x_2\in \mathcal Y_C^\Pi(\tilde \Omega)$.}
\end{equation}
Set
\[
\mathcal F:=\mathcal T\circ \mathcal G \colon \mathcal Y_C^\Pi(\tilde \Omega)\to \mathcal Y_C^\Pi(\tilde \Omega).
\]
By~\eqref{contr1} and since $\rho \|\mathcal T\|<1$, we conclude that $\mathcal F$ is a contraction on $\mathcal Y_C^\Pi(\tilde \Omega)$. Consequently, there exists a unique $x\in \mathcal Y_C^\Pi(\tilde \Omega)$ such that $x=\mathcal T (\mathcal G x)$. Taking into account the definition of $\mathcal T$, we see that 
\[
\begin{split}
 x(\ell, \w, n+1)&=\Phi(1, n+\ell, \theta^n \w)x(\ell, \w, n)+\Psi_\Phi(1, n+\ell, \theta^n \w)x(\ell, \w, n)\\
 &\phantom{=}+\varphi_{n+\ell+1}^{-1}y(\ell, \w, n+1),
\end{split}
\]
which implies that 
\[
x(\ell, \w, n+1)=\Psi(1, n+\ell, \theta^n \w)x(\ell, \w, n)+\varphi_{n+\ell+1}^{-1}y(\ell, \w, n+1),
\]
for $(\ell, \w, n)\in \Z^+\times \tilde \Omega \times \Z^+$. Hence, $x$ satisfies~\eqref{ADM} with $\Psi$ instead of $\Phi$. The uniqueness of $x$ follows from the uniqueness of the fixed point of $\mathcal F$.

Next, it is easy to show that for $\lw \in \TWW$ and $n\in \Z^+$ that 
\[
\begin{split}
\Psi(n, \ell, \w) &=\Phi(n, \ell, \w)\\
&\phantom{=}+\sum_{j=0}^{n-1}\Phi(n-1-j, j+\ell+1, \theta^{j+1}\w)\Psi_\Phi(1, j+\ell, \theta^j \w)
\Psi(j, \ell, \w).
\end{split}
\]
Hence, for arbitrary $v\in X$ we have that 
\[
\begin{split}
&\|\Psi(n, \ell, \w)v\|_{\Theta^n \lw} \\
&\le M\lw \left (\frac{\mu_{\ell+n}}{\mu_\ell}\right )^{\lambda \lw}\|v\|_{\lw}\\
&\phantom{\le}+M\lw c\lw \sum_{j=0}^{n-1}\left (\frac{\mu_{\ell+n}}{\mu_{j+\ell+1}}\right )^{\lambda \lw}\varphi_{j+\ell+1}^{-1}\|\Psi(j, \ell, \w)v\|_{\Theta^j \lw}.
\end{split}
\]
By using discrete Gronwall's inequality and~\eqref{prop:ESTIMATES-est-=1} we conclude that 
\[
\begin{split}
\|\Psi(n, \ell, \w)v\|_{\Theta^n \lw } &\le M\lw \left (\frac{\mu_{\ell+n}}{\mu_\ell}\right )^{\lambda \lw} e^{M\lw c\lw \sum_{j=\ell+1}^{n+\ell}\varphi_j^{-1}}\|v\|_{\lw} \\
&\le M\lw \left (\frac{\mu_{\ell+n}}{\mu_\ell}\right )^{\lambda \lw}\left (\frac{\mu_{\ell+n+1}}{\mu_{\ell+1}}\right )^{\eta M\lw c\lw}\|v\|_{\lw} \\
&\le M\lw \left (\frac{\mu_{\ell+n}}{\mu_\ell}\right )^{\lambda \lw}\left (\frac{\mu_{\ell+n+1}}{\mu_{\ell}}\right )^{\frac 1 2 \eta }\|v\|_{\lw} \\
&\le M\lw \eta^{\frac 1 2 \eta}\left (\frac{\mu_{\ell+n}}{\mu_{\ell}}\right )^{\lambda \lw +\frac 1 2 \eta}\|v\|_{\lw}.
\end{split}
\]
We conclude that there exists a  $\Theta$-forward invariant random variable $N\colon \TWW \to [1, \infty)$ such that 
\[
\|\Psi (n, \ell, \w)v\|_{\Theta^n \lw}\le N\lw \left(\frac{\mu_{\ell+n}}{\mu_\ell}\right)^{\lambda\lw +\frac 1 2 \eta}\|v\|_{\lw},
\]
for $\lw \in \TWW$ and $v\in X$.
By Theorem~\ref{T52} we conclude that $\Psi$ admits $\mu$-dichotomy with respect to the random norm $\mathcal N$.

\end{proof}

	\section{Relation with random nonuniform $(\mu,\nu)$-dichotomy}\label{se:mu_vs_munu}

Let $\mu,\nu\colon\Z^+\to\R^{+}$
	be growth rates, with $\nu_n\ge 1$ for each $n\in \Z^+$, and let $\mathcal N$ be a random norm. We say that a linear
	NRDS $\Phi$ admits a \textit{random nonuniform $(\mu,\nu)$-dichotomy} if there exist a $\theta$-invariant subset $\tilde{\Omega}\subset\Omega$ of full measure, a family of projections $
P=\{P_{(\ell,\omega)}: \lw \in \TWW\}$ satisfying $P_1)-P_4)$ and
	 $\Theta$-forward invariant random variables $\lambda, D, \varepsilon\colon\TWW\to(0,+\infty)$
such that for every $(n,\ell,\w)\in\Z^+\x\Z^+\x\tilde\W$  we have
	\begin{equation}
		\label{eq:dich-mu-nu-1}\|\Phi(n,\ell,\w) P_{\lw}\|\le D\lw \left( \dfrac{\mu_{\ell+n}}
		{\mu_\ell}\right)^{-\lambda\lw}\nu_\ell^{\varepsilon\lw},
	\end{equation}
	and for $\lw \in \TWW$ and $0\le n\le \ell$
	\begin{equation}
		\label{eq:dich-mu-nu-2}\|\Phi(-n,\ell, \w)Q_{\lw }\|\le D\lw \left( \dfrac
		{\mu_{\ell}}{\mu_{\ell-n}}\right)^{-\lambda\lw}\nu_\ell^{\varepsilon\lw}. 
\end{equation}

\begin{remark}	Notice that~\eqref{eq:dich-mu-nu-2} is equivalent to say that for all
$\lw\in\TWW$ and $0\leq n\leq \ell$ we have
\begin{equation*}
		\|\Phi(-n,\ell+n,\theta^n\w)Q_{\Theta^n(\ell,\w)}\|\le 
		D\lw\left(\frac{\mu_{\ell+n}}{\mu_\ell}\right)^{-\lambda\lw}\nu_\ell^{\varepsilon\lw}.
  	\end{equation*}
\end{remark}


If
	additionally there are $\Theta$-forward invariant  random variables $K,b,\gamma\colon\TWW\to(0,+\infty)$ with $b\ge \lambda$   such that
	\begin{equation}
		\label{eq:nonunif-strong-dich}\|\Phi(n,\ell,\w)Q_{\lw}\| \le K\lw \left(\frac{\mu_{\ell+n}}{\mu_\ell}\right
		)^{b\lw} \nu_\ell^{\gamma\lw},
	\end{equation}
	we say that $\Phi$   admits
	a \textit{strong random nonuniform $(\mu,\nu)$-dichotomy}.

\begin{example}
We have an \textit{exponential random nonuniform dichotomy}  with respect to a random $\mathcal N$ if it admits a $(\mu,\nu)$-dichotomy with $\mu_n=\nu_n=\e^{n}$. Conditions~\eqref{eq:dich-mu-nu-1} and~\eqref{eq:dich-mu-nu-2} become
\begin{equation*}
		\|\Phi(n,\ell,\w) P_{\lw}\|\le D\lw e^{-\lambda\lw n+\ell\varepsilon\lw}
	\end{equation*}
	for every $(n,\ell,\w)\in\Z^+\x\Z^+\x\tilde\W$ , and for $\lw \in \TWW$ and $0\le n\le \ell$,
	\begin{equation*}
		\| \Phi(-n,\ell, \w)Q_{\lw }\|\le D\lw e^{\lambda\lw n+\ell\varepsilon\lw}. 
\end{equation*}
For the polinomial case, we may consider $\mu_n=\nu_n=n+1$, and for a logarithmic-type of nonuniform $(\mu,\nu)$-dichotomy we can take $\mu_n = \nu_n= \log(n+1)$.

Example~\ref{ex:exist:dich} can be easily adapted to obtain random nonuniform $(\mu,\nu)$-dichotomies for given growth rates $\mu$ and $\nu$. 

\end{example}

The following is a central result of this section.
 \begin{theorem}
The following properties are equivalent:
\begin{enumerate}
\item[(i)] $\Phi$ admits a strong random nonuniform $(\mu, \nu)$-dichotomy;
\item[(ii)] $\Phi$ admits a $\mu$-dichotomy with respect to a random norm $\mathcal{N}$, with the property that there exist a $\theta$-invariant subset $\tilde{\Omega} \subset \Omega$ of full measure and $\Theta$-forward invariant random variables $\bar{K}, \bar{\varepsilon} \colon \TWW \to (0, +\infty)$ such that, for all $\lw \in \TWW$, we have
\begin{equation}\label{LN1}
    \|v\| \le \|v\|_ {\lw} \le \bar K\lw \nu_\ell^{\bar {\varepsilon} \lw}\|v\| .
\end{equation}
Moreover, there are  $\Theta$-forward invariant random variables $\bar{\lambda}\colon\TWW\to(0,+\infty)$ and $M\colon\TWW\to[1,+\infty)$ such that
\begin{equation}
			\label{boundedgrowth}
                \|\Phi(n,\ell,\w)v\|_{\Theta^n\lw}\le
			M(\ell,\w)\left(\frac{\mu_{\ell+n}}{\mu_\ell}\right)^{\bar{\lambda}\lw}\|v\|_{\lw},
		\end{equation}
		for all $(n,\ell,\w)\in\Z^+\x \Z^+    \x\tilde\W$ and $v\in X$.
\end{enumerate}

 \end{theorem}

 \begin{proof}
$(i)\implies (ii)$. Suppose that $\Phi$ admits a strong random nonuniform $(\mu,\nu)$-dichotomy. For $\lw \in \TWW$ and $v\in X$, set
\[
\|v\|_{\lw}:=\|v\|_{\lw}^{(s)}+\|v\|_{\lw}^{(u)},
\]
where
\[
\|v\|_{\lw}^{(s)}:=\sup_{n\ge 0} \left (\|\Phi(n, \ell, \w)P_{\lw}v\|  \left (\frac{\mu_{\ell+n}}{\mu_\ell}\right )^{\lambda \lw} \right )
\]
and 
\[
\begin{split}
\|v\|_{\lw}^{(u)} :&=\sup_{0\le n\le \ell}\left(\|\Phi(-n,\ell, \w)Q_{\lw }v\|\left( \dfrac
		{\mu_{\ell}}{\mu_{\ell-n}}\right)^{\lambda\lw}\right)\\
  &\phantom{=}+\sup_{n\ge 1}\left (\|\Phi(n, \ell, \w)Q_{\lw}v\|\left(\frac{\mu_{\ell+n}}{\mu_\ell}\right
		)^{-b\lw} \right).
  \end{split}
\]
We observe that it follows from~\eqref{eq:dich-mu-nu-1}, \eqref{eq:dich-mu-nu-2} and~\eqref{eq:nonunif-strong-dich} that 
\[
\|v\|_{\lw}\le 2D\lw \nu_{\ell}^{\varepsilon \lw}\|v\|+K\lw \nu_\ell^{\gamma \lw}\|v\|.
\]
Since, on the other hand
\[
\|v\|_{\lw}\ge \|P_{\lw}v\| +\|Q_{\lw}v\| \ge \|v\|,
\]
we conclude that~\eqref{LN1} holds with
\[
\bar{\varepsilon} \lw:=\max \{\varepsilon \lw , \gamma \lw \} \quad \text{and} \quad \bar K \lw:=2D\lw +K\lw.
\]
Next, for $(n,\ell,\w) \in \Z^+\times\Z^+\times \tilde \Omega$ and $v\in X$, we have that 
\[
\begin{split}
&\|\Phi(n, \ell, \w)P_{\lw} v\|_{\Theta^n \lw}  \\
&=
\sup_{m\ge 0}\left( \|\Phi(m, \ell+n, \theta^n \w)\Phi(n, \ell, \w)P_{\lw}v\|
\left (\frac{\mu_{\ell+n+m}}{\mu_{\ell+n}} \right )^{\lambda \lw}\right )
 \\
&=\sup_{m\ge 0} \left(  \|\Phi(m+n, \ell, \w)P_{\lw}v\| 
\left (\frac{\mu_{\ell+n+m}}{\mu_{\ell+n}} \right )^{\lambda \lw}\right ) \\
&=\left (\frac{\mu_{\ell+n}}{\mu_\ell}\right )^{-\lambda \lw}\sup_{m\ge 0} \left(  \|\Phi(m+n, \ell, \w)P_{\lw}v\| 
\left (\frac{\mu_{\ell+n+m}}{\mu_{\ell}} \right )^{\lambda \lw}\right ) \\
&\le \left (\frac{\mu_{\ell+n}}{\mu_\ell}\right )^{-\lambda \lw}\|v\|_{\lw}.
\end{split}
\]
Therefore, 
\begin{equation}\label{11}
\|\Phi(n, \ell, \w)P_{\lw} v\|_{\Theta^n \lw} \le \left (\frac{\mu_{\ell+n}}{\mu_\ell}\right )^{-\lambda \lw}\|v\|_{\lw}, 
\end{equation}
for $(n,\ell, \w)\in \Z^+\times \Z^+\times \tilde \Omega$ and $v\in X$.
Moreover,  for $n\le \ell$ we have (using that $b\ge \lambda$) that 
\begin{equation}\label{calc}
\begin{split}
&\|\Phi(-n, \ell, \w)Q_{\lw}v\|_{(\ell-n, \theta^{-n}\w)} \\
&=\sup_{0\le m\le \ell-n}\left (\|
\Phi(-m, \ell-n, \theta^{-n}\w)\Phi(-n, \ell, \w)Q_{\lw}v\| \left (\frac{\mu_{\ell-n}}{\mu_{\ell-n-m}}\right)^{\lambda \lw} \right )
\\
&\phantom{=}+\sup_{m\ge 1}\left (
\|\Phi(m, \ell-n, \theta^{-n}\w)\Phi(-n, \ell, \w)Q_{\lw}v\| \left (\frac{\mu_{\ell-n+m}}{\mu_{\ell-n}}
\right)^{-b \lw}
\right )\\
&=\sup_{0\le m\le \ell-n}\left (\|
\Phi(-(m+n), \ell, \w)Q_{\lw}v\| \left (\frac{\mu_{\ell-n}}{\mu_{\ell-n-m}}\right)^{\lambda \lw} \right )
\\
&\phantom{=}+\sup_{m\ge 1}\left (
\|\Phi(m-n, \ell, \w)Q_{\lw}v\| \left (\frac{\mu_{\ell-n+m}}{\mu_{\ell-n}}
\right)^{-b \lw}\right )\\
&\le \left (\frac{\mu_\ell}{\mu_{\ell-n}}\right )^{-\lambda \lw} \sup_{n\le k\le \ell}\left (\|
\Phi(-k, \ell, \w)Q_{\lw}v\| \left (\frac{\mu_{\ell}}{\mu_{\ell-k}}\right)^{\lambda \lw} \right )\\
&\phantom{\le}+\sup_{1\le m\le n}\left (
\|\Phi(m-n, \ell, \w)Q_{\lw}v\| \left (\frac{\mu_{\ell-n+m}}{\mu_{\ell-n}}
\right)^{-\lambda \lw}\right )\\
&\phantom{\le}+\sup_{m>n}\left (
\|\Phi(m-n, \ell, \w)Q_{\lw}v\| \left (\frac{\mu_{\ell-n+m}}{\mu_{\ell-n}}
\right)^{-b \lw}\right )\\
&=\left (\frac{\mu_\ell}{\mu_{\ell-n}}\right )^{-\lambda \lw} \sup_{n\le k\le \ell}\left (\|
\Phi(-k, \ell, \w)Q_{\lw}v\| \left (\frac{\mu_{\ell}}{\mu_{\ell-k}}\right)^{\lambda \lw} \right )\\
&\phantom{=}+\left (\frac{\mu_\ell}{\mu_{\ell-n}}\right )^{-\lambda \lw}\sup_{0\le k\le n-1}\left (\| \Phi(-k, \ell, \w)Q_{\lw}v\|\left (\frac{\mu_{\ell-k}}{\mu_{\ell}}\right )^{-\lambda \lw}\right )\\
&\phantom{=}+\left (\frac{\mu_\ell}{\mu_{\ell-n}}\right)^{-\lambda \lw}
\sup_{k\ge 1}
\left (
\|\Phi(k, \ell, \w)Q_{\lw}v\| \left (\frac{\mu_{\ell+k}}{\mu_{\ell}}
\right)^{-b \lw}\right )\\
&\le 2\left (\frac{\mu_\ell}{\mu_{\ell-n}}\right )^{-\lambda \lw}\|v\|_{\lw}.
\end{split}
\end{equation}
Consequently,
\begin{equation}\label{22}
    \|\Phi(-n, \ell, \w)Q_{\lw}v\|_{(\ell-n, \theta^{-n}\w)}\le 2\left (\frac{\mu_\ell}{\mu_{\ell-n}}\right )^{-\lambda \lw}\|v\|_{\lw},  
\end{equation}
for all $(n, \ell, \w)\in \Z^+\times \Z^+\times \tilde \Omega$, with $n\le \ell$, and $v\in X$. From~\eqref{11} and~\eqref{22} it follows that $\Phi$ admits $\mu$-dichotomy with respect to the random norm $\mathcal N$ given by $\mathcal N(\lw,\cdot)=\| \cdot \|_{\lw}$ for all $\lw\in\TW$.

We now establish~\eqref{boundedgrowth}. Proceeding similarly to~\eqref{calc}, one finds that 
\begin{equation}\label{33}
    \|\Phi(n, \ell, \w)Q_{\lw}v\|_{\Theta^n \lw} \le 2\left (\frac{\mu_{\ell+n}}{\mu_\ell}\right )^{b \lw}\|v\|_{\lw},
\end{equation}
for all $(n, \ell, \w)\in \Z^+\times \Z^+\times \tilde \Omega$ and $v\in X$. By~\eqref{11} and~\eqref{33} we conclude that~\eqref{boundedgrowth} holds with $M \lw =3$ and $\overline \lambda=b$.

$(ii) \implies (i)$ By~\eqref{eq:dich-1} and~\eqref{LN1} we have for $(n,\ell, \w)\in \Z^+\times \Z^+\times \tilde \Omega$ and $v\in X$ that 
\[
\begin{split}
\|\Phi(n,\ell,\w) P_{\lw} v\| &\le \|\Phi(n,\ell,\w) P_{\lw} v\|_{\Theta^n(\ell,\w)} \\
&\le D\lw\left(\frac{\mu_{\ell+n}}{\mu_\ell}\right)^{-\lambda\lw}\|v\|_{\lw} \\
&\le D\lw  \bar K \lw \left(\frac{\mu_{\ell+n}}{\mu_\ell}\right)^{-\lambda\lw}\nu_\ell^{\bar {\varepsilon} \lw} \|v\|,
\end{split}
\]
yielding~\eqref{eq:dich-mu-nu-1}. Similarly, \eqref{eq:dich-2-b} and~\eqref{LN1} give that
\[
\begin{split}
\|\Phi(-n,\ell, \w)Q_{\lw}v\| &\le \|\Phi(-n,\ell,\w)Q_{\lw}v\|_{\Theta^n \lw} \\
&\le D\lw \left(\frac{\mu_\ell}{\mu_{\ell-n}}\right)^{-\lambda \lw}\|v\|_{\lw}\\
&\le D\lw \bar K\lw \left(\frac{\mu_\ell}{\mu_{\ell-n}}\right)^{-\lambda \lw}\nu_\ell^{\bar {\varepsilon} \lw} \|v\|,
\end{split}
\]
for $\lw \in \TWW$, $0\le n\le \ell$ and $v\in X$. Hence, \eqref{eq:dich-mu-nu-2} holds.
Finally, \eqref{LN1}, \eqref{boundedgrowth} together with~\eqref{eq:dich-2} (applied for $n=0$)
\[
\begin{split}
    \|\Phi(n,\ell,\w)Q_{\lw}v\| &\le \|\Phi(n,\ell,\w)Q_{\lw}v\|_{\Theta^n \lw} \\
    &\le M(\ell,\w)\left(\frac{\mu_{\ell+n}}{\mu_\ell}\right)^{\bar{\lambda}\lw}\|Q_{\lw}v\|_{\lw} \\
    &\le D\lw M(\ell,\w)\left(\frac{\mu_{\ell+n}}{\mu_\ell}\right)^{\bar{\lambda}\lw}\|v\|_{\lw}\\
    &\le  D\lw M(\ell,\w)\bar K \lw \left(\frac{\mu_{\ell+n}}{\mu_\ell}\right)^{\bar{\lambda}\lw}\nu_\ell^{\bar {\varepsilon} \lw} \|v\|,
\end{split}
\]
for $(n, \ell, \w)\in \Z^+\times \Z^+\times \tilde \Omega$ and $v\in X$. Consequently, \eqref{eq:nonunif-strong-dich} holds. The proof of the theorem is completed.
\end{proof}

  \section*{Acknowledgments}
D.~Dragi\v cevi\' c was supported in part by University of Rijeka under the project uniri-iskusni-prirod-23-98 3046. C.~Silva and H. Vilarinho were partially supported by FCT through CMA-UBI (project UIDB/MAT/00212/2020).

\end{document}